\documentclass[3p,reqno]{amsart}
\usepackage{tikz}
\usepackage{hyperref}
\usepackage{latexsym,amssymb,euscript,mathrsfs}
\usepackage{amsthm}
\usepackage{amsmath}
\usepgflibrary{arrows}

\newtheorem{theorem}{Theorem}[section]
\newtheorem{lemma}[theorem]{Lemma}
\newtheorem{proposition}[theorem]{Proposition}
\newtheorem{corr}[theorem]{Corollary}

\theoremstyle{remark}

\theoremstyle{definition}
\newtheorem{definition}[theorem]{Definition}

\numberwithin{equation}{section}
\newcommand{\cn}{\mathbb{C}}
\newcommand{\nn}{\mathbb{N}}
\newcommand{\rn}{\mathbb{R}}
\newcommand{\zn}{\mathbb{Z}}

\begin{document}

\title{Maximal Ergodic Inequalities for Banach Function Spaces}

\date{\today}

\author{Richard de Beer}
\email{richardjohndebeer@gmail.com}

\author{Louis Labuschagne}
\email{louis.labuschagne@nwu.ac.za}

\address{DST-NRF CoE in Math. and Stat. Sci, Unit for BMI, Internal Box 209, School of Comp., Stat. \& Math. Sci., NWU, Pvt. Bag X6001, 2520 Potchefstroom, South
Africa} 

\thanks{The contributions of the second author are based upon research supported by the National Research Foundation. Any opinion, findings and conclusions or recommendations expressed in this material, are those of the authors, and therefore the NRF do not accept any liability in regard thereto.}

\begin{abstract} We analyse the Transfer Principle, which is used to generate weak type maximal inequalities for ergodic operators, and extend it to the general case of $\sigma$-compact locally compact Hausdorff groups acting measure-preservingly on $\sigma$-finite measure spaces. We show how the techniques developed here generate various weak type maximal inequalities on different Banach function spaces, and how the properties of these function spaces influence the weak type inequalities that can be obtained. Next we demonstrate how the techniques developed imply almost sure pointwise convergence of a wide class of ergodic averages. In closing we briefly indicate the utility of these results for Statistical Physics.
\end{abstract}

\medskip



\maketitle

{\small \tableofcontents}

\section{Introduction}
 
Pointwise ergodic theorems have had an illustrious history spanning over 80 years since G.D. Birkhoff first proved the foundational result in 1931. The proof of his ergodic theorem has been so refined that one can give an elementary, leisurely  demonstration in about two pages \cite{kepe}. However to work with more general ergodic averages, it seems one must still rely on a different approach.
This is the technique of maximal operators. The idea is that once one estimates the behaviour of these maximal operators, proving the ergodic theorems becomes quite simple. (We explain the proof strategy in Section \ref{S:appl} in the form of a three-step programme.) 
Wiener \cite{wie} developed a method, later greatly embellished by Calder{\'o}n \cite{ca}, for computing the requisite properties of the maximal operators, a method that is the central theme of this work: the Transfer Principle. It is our goal to extend the scope of this Principle and hence the scope of the maximal operator technique in proving pointwise ergodic theorems.

Broadly speaking, a dynamical system consists of three elements: a measure space $(\Omega,\mu)$, a topological group $G$, and an action $\alpha$, continuous in some sense, that binds them together by mapping $G$ into the group of invertible measure-preserving transformations of $\Omega$. The Transfer Principle refers to a body of techniques that allow one to transform certain types of operators acting on function spaces over $G$ to corresponding transferred operators acting on function spaces over $\Omega$, in such a way that many essential properties of the operator are preserved.

We have three aims: firstly, to identify a class of \emph{transferable operators} that is specific enough to allow for transference in a very general class of spaces, yet general enough to encompass all the important applications of the Transfer Principle, secondly to broaden the reach of the techniques used to determine the weak type of the transferred operator to a wider class of spaces, and finally to outline how these results may be used to derive pointwise ergodic theorems.

We introduce the concept of \emph{transferable} operators in quite a general setting (Definition \ref{D:to}), before finally introducing the concept of a transferred operator $T^\#$ in Definition \ref{R:transfer} (after some much needed preparation). The determination of the weak type of the transferred operator - call it $T^\#$- rests on results requiring $\Omega$ to be countably generated and resonant as defined just before Proposition \ref{P:XinY(Z)}. This fulfills the first aim.

Computing the weak type of a transferred operator $T^\#$ is achieved with Corollary \ref{C:CalderonTh1corr2}. A noteworthy feature of this result is that it shows that the most important factors determining the weak type of $T^\#$ are the fundamental functions associated with the function spaces defining the weak type of $T$. In practice, however, one must often deal with sequences of transferable operators. Computing the weak type of the transfer of the limit of such a sequence, requires more delicate analysis. We present our results in Theorem \ref{P:zyzz} and Theorems \ref{C:endcal9}, \ref{C:endcal}, and \ref{C:endcal2}. Of these, Theorems \ref{C:endcal9}, \ref{C:endcal} and \ref{C:endcal2} are extensions of \cite[Theorem 1]{ca} to the case of general amenable locally compact groups acting on measure spaces $\Omega$ that may possibly have unbounded measure.

In this way, the problem of computing the weak type of the transferred operator is reduced to computations involving certain well-behaved real-valued functions. In particular, these results allow us to estimate the weak type of the maximal operator associated with ergodic averages over a wide class of rearrangement invariant Banach function spaces. This completes the second aim.

Finally, we address the third aim by proving pointwise ergodic theorems. This is achieved by transferring information obtained using Fourier analysis on the group, to properties of the ergodic averages, and then combining this with information regarding the function space on which they act, as encoded in the fundamental function of that space. The main results are Theorem \ref{theoBirk} and Corollary \ref{corrInterpBirk}.

The importance of the Transfer Principle in ergodic theory has long been appreciated - see the excellent overview given in \cite{abe}. Apart from Calder\'on's seminal paper \cite{ca}, this principle is treated in some detail in the monograph \cite{cowe_tr} and employed extensively in \cite{rowi}. In \cite{cowe_tr} the authors discuss the transfer of convolution operators on locally compact groups by means of an action on a general measure space. In their study they remain within the category of $L^p$-spaces. In \cite{pa} the author makes use of Orlicz spaces to prove results about the pointwise convergence of ergodic averages along certain subsets of the natural numbers. In \cite{fiwa} the Transfer Principle of Coifman and Weiss is extended to weighted Orlicz spaces for group actions that are uniformly bounded in a sense determined by the space. Other important contributions regarding the development of the transfer principle include the work of Haase \cite{haa}, Lin and Wittmann \cite{liwi}, and of Asmar, Berkson and Gillespie (\cite{asbegi1} and \cite{asbegi2}). However with each of these the focus of the research is somewhat different to that of the present paper. Our focus is to find a general formalism which will allow for the extension of the transfer principle to as wide a class of rearrangement invariant Banach function spaces as possible for group actions by general locally compact groups. Haase on the other hand demonstrates how the transfer principle can be extended to \emph{semigroups} of operators as opposed to groups. Although his ideas are potentially more widely applicable, almost all his results were obtained in the category of $L^p$-spaces. Lin and Wittmann study the ergodicity of a sequence of probabilities on a locally compact $\sigma$-compact group, and show the existence of such a sequence to be equivalent to amenability. In the process they obtain an extension of Calderon's transfer principle for the case of Lamperti operators. Asmar, Berkson and Gillespie also do not depart from the category of $L^p$-spaces. In \cite{asbegi1} strong-type transfer results are obtained, with \cite{asbegi2} dealing with the case of weak-type transfer results. Their strategy for obtaining these weak type results is to use the device of \emph{distributional boundedness} rather than the concept of \emph{transferable operators}, which we employ in the present work. All of these various strategies have one thing in common: to find appropriate formalisms that allow for the extension of Calder\'on's work. However as can clearly be seen from the above discussion, there are many possibly divergent ways to extend the work of Calder\'on, each important in their own right.

\begin{definition}\label{D:dynsys} A  dynamical system consists of a \textit{multiplicative} locally compact group $G$ acting on a measure space $(\Omega,\mu)$ via an action of $G$ on $\Omega$ by $\alpha$. The action is measure-preserving in the sense that for any measurable subset $A\subseteq\Omega$ and $g\in G$, $\mu(\alpha^{-1}_g(A))=\mu(A)$. Furthermore, the map $\tilde{\alpha}:\Omega\times G\to\Omega:t\times\omega\mapsto\alpha_t(\omega)$ is also measurable. The data is summarised by these four objects:  \begin{equation*}\label{E:dynsys}(\Omega,\mu,G,\alpha).\end{equation*} \end{definition}

Unless otherwise stated, we will throughout this paper assume $(\Omega,\mu)$ and the group $G$ to be $\sigma$-finite. Note also that for locally compact groups, $\sigma$-finiteness is the same as $\sigma$-compactness \cite[\S 2.3]{fo}.

The condition that $\tilde{\alpha}:\Omega\times G\to\Omega:(\omega,t)\mapsto\alpha_t(\omega)$ be measurable is equivalent to stating that if $f$ is a measurable function of $\Omega$, then the function $F(\omega,t):=f(\alpha_t(\omega))$ is measurable on $\Omega\times G$ because $F=f\circ\widetilde{\alpha}$. With a slight abuse of notation, for any measurable function $f$ on $\Omega$ and $t\in G$ we can then define $\alpha_t(f)$ by setting $\alpha_t(f)(\omega):=f(\alpha_t(\omega))$ for a $\mu$-almost all $\omega\in\Omega$. 

We shall denote the right Haar measure on $G$ by the symbol $h$. 

One final notational convention: if $A$ is a measurable subset of a measure space $(\Omega,\mu)$, we shall for brevity write $|A|:=\mu(A)$. Likewise, if $K$ is a measurable subset of the locally compact group $G$, we shall denote the right Haar measure of $K$ by $|K|$.

Let us briefly describe the organisation of the paper. In Section \ref{S:2}, we define the Transfer Principle and analyse it in some detail. This involves quite intricate measure-theoretic considerations, including the development of a theory of locally Bochner integrable functions in parallel with the classical theory of Bochner integrable functions.  

In Section \ref{S:BFS} we bring to mind some basic constructions and definitions in the theory of rearrangement invariant Banach function spaces. We emphasise how in the general theory a central role is played by the fundamental function of such spaces, and how a great deal of their structure and behaviour is reflected in this function.  We also estimate some integrals that arise naturally for functions on product spaces (Proposition \ref{P:XinY(Z)}).

Section \ref{S:weakTypeDef} contains the main results for estimating the weak type of the transferred operator, namely Corollary \ref{C:CalderonTh1corr2} and Theorems \ref{C:endcal9}, \ref{C:endcal} and \ref{C:endcal2}. This Section is based on the work of the previous two sections and an extension of an inequality of Kolmogorov (Theorem \ref{thKol}). 

The final two sections contain derivations of pointwise ergodic theorems and applications to Orlicz spaces arising in statistical physics. We show how properties of the function spaces and the transfer operators combine to determine a variety of ergodic theorems.

\section{The Transfer Principle}\label{S:2}

\subsection{Construction and measure theoretic considerations}

The first order of business is to specify what the Transfer Principle is and to which operators the procedure applies. This requires a careful analysis of certain unusual locally convex spaces of measurable functions, and transformations acting upon them.

Recall that an operator $T$ whose domain is some linear subspace of the measurable functions on $(\Omega,\mu)$ and mapping into the measurable functions on a measure space $(\Omega_1,\mu_1)$, is said to be \textit{sublinear} if for any $f$ and $g$ in the domain of $T$ and complex $\lambda$, we have $|T(f+g)|\leq|T(f)|+|T(g)|$ and $|T(\lambda f)|=|\lambda||T(f)|$. By the term \emph{positive-valued}, we mean a sublinear map satisfying $T(f)=|T(f)|$ for all $f$. 

By $C(G)$ we mean the space of continuous functions on $G$, topologised with the compact-open topology given by the seminorms $q_K(f)=\sup\{|f(t)|:t\in K\}$ as $K$ ranges over all compact subsets of $G$.

A transformation $T:\mathcal{E}\to \mathcal{F}$ between locally convex spaces $\mathcal{E}$ and $\mathcal{F}$ is called \emph{quasi-bounded} if for some set of seminorms $\{p_\alpha\}$ determining the topology of $\mathcal{F}$, we can find a set of seminorms $\{q_\beta\}$ determining the topology of $\mathcal{E}$ and a positive scalar $c$, such that for any $\alpha$ we can find an $\alpha'$ such that $p_\alpha(T(x)-T(y))\leq c q_{\alpha'}(x-y)$ for all $x,y\in \mathcal{E}$. Clearly, quasi-bounded operators are continuous.

The converse of this statement is also true in some important cases. On the space of measurable functions over a $\sigma$-finite measure space $(\Omega,\mu)$ we can define a family of seminorms $p_A(f):=\int_A|f|~d\mu$ where $A$ ranges over all subsets of $\Omega$ of finite measure. Those measurable functions for which $p_A(f)$ is finite for all $A$ of finite measure, form the  locally convex vector space of \textit{locally integrable functions}. This space is denoted by $L^{\rm loc}(\Omega)$. Returning to the issue of quasi-boundedness it follows in particular that any continuous operator $T:L^{\rm loc}(G)\to C(G)$ which is either linear or positive-valued sublinear, will automatically be quasi-bounded. We proceed to verify this claim.

It is well-known that every locally convex space admits a neighbourhood base at 0 consisting of absorbent absolutely convex sets. Hence continuity of the sublinear map $T$ at 0, amounts to the claim that for every $\epsilon>0$ and compact $K\subset G$, we can find an absorbent absolutely convex neighbourhood $U\subset L^{\rm loc}(G)$ of 0, such that $T(U)\subset V=\{f\in C(G):\sup\{|f(t)|:t\in K\}<\epsilon\}$. Now let $p_U$ and $p_V$ be the Minkowski functionals respectively corresponding to $U$ and $V$. Then the inclusion $T(U)\subset V$ forces the inequality \begin{equation}\label{eq:seminormcont} p_V(T(f))\leq p_U(f)\quad (f\in L^{\rm loc}(G)).\end{equation} To see this observe that if this inequality did not hold, we would be able to find some $f\in L^{\rm loc}(G)$ for which $p_V(T(f)) > p_U(f)$. On rescaling $f$ if necessary, we may arrange matters so that $p_V(T(f))>1> p_U(f)$.  This in turn amounts to the claim that $f\in U$ and $T(f)\not\in V$ -- a clear contradiction. 

Next let $f,g \in L^{\rm loc}(G)$ be given. From the inequalities
$|T(g)|\leq |T(f)|+|T(g-f)|$ and $|T(f)|\leq |T(f-g)|+|T(g)|$, it follows that $\left||T(f)|-|T(g)|\right|\leq |T(f-g)|.$
The  definition of $p_V$ implies that it is order-preserving on $C(G)$, and so $$p_V(T(f)-T(g))=p_V(|T(f)|-|T(g)|)\leq p_V(|T(f-g)|)= p_V(T(f-g)).$$ The claim follows by combining this inequality with inequality \ref{eq:seminormcont}.  

One can generalise this discussion to prove that continuous \emph{quasi}-linear positive-valued maps are also quasi-bounded, but omit this detail as we will not work with such maps in this work.

\begin{definition}\label{D:to} A quasi-bounded operator $T:L^{\rm loc}(G)\to C(G)$ is called a \emph{transferable} operator if 
\begin{enumerate}
\item $T$ is quasi-bounded;
\item $T$ is \textbf{semilocal} in that there exists an open neighbourhood $U$ of $1\in G$ with compact closure such that if ${\rm supp}(f)$ is contained in a set $V$, then ${\rm supp}(T(f))$ is contained in  $VU$;
\item $T$ is \textbf{right translation invariant} in that for all $t\in G$ and $f\in L^{\rm loc}(G)$, $$\tau_t\circ T(f)=T\circ\tau_t(f),$$ where $\tau_t(f)$ is the operator defined by $s\mapsto f(st)$ for all measurable $f$ and $s\in G$.
\end{enumerate}
\end{definition}

Starting with a quasi-bounded operator $T$ which acts on functions over $G$, our goal is to define the \emph{transfer} $T^\#$ of the operator, acting on functions over $\Omega$.  Starting with a function $f\in L^{1+\infty}(\Omega)$, Calder{\'o}n's original conception of the transfer operator was to first use the effect of the group action on $f$ to define an associated function of two variables $F(\omega,t):=f(\alpha_t(\omega))$. Next he defined an action of $T$ on this space of functions of two variables by applying $T$ to the cross-sections $F_\omega$ of $F$ at $\omega\in\Omega$, to produce a new function $F'(\omega,t):=(T(F_{\omega}))(t)$. Finally he then set the $T^{\#}(f)$ to be $F'(\omega,1)$. Despite the elegance of this construction, it is not a priori clear that the output as defined above will indeed be measurable. Hence we will carefully analyse each of the component parts of this construction, before ultimately showing in Definition \ref{R:transfer} how they can all be put together to produce a well-defined transfer operator 
$f\to T^\#(f)$. The extent to which the eventual realisation of the transferred operator agrees with Calder{\'o}n's original conception thereof, will also briefly be investigated. Specifically we need to 
\begin{itemize}
\item investigate the properties of the map $f\to F$ where $F(\omega,t):=f(\alpha_t(\omega))$;
\item find a natural home for the functions produced on the manner described above and find a way to extend the action of $T$ to that space;
\item denoting the extension of $T$ by $\widetilde{T}$, we need to show that $\widetilde{T}$ can be constructed in such a way that the partial point evaluation $\widetilde{T}(F)(\cdot, \cdot)\to \widetilde{T}(F)(\cdot, 1)$ appears as a continuous operator on the range space of $\widetilde{T}$, mapping into $L^{\rm loc}(\Omega)$.
\end{itemize} 

We commence with this programme by clarifying the completeness of the space of locally integrable functions.
We show this by using the technique of projective limits. As the same method is used to prove all the subsequent completeness results, we shall provide the demonstration in full here.

\begin{lemma}\label{L:projLim} If $(\Omega,\mu)$ is $\sigma$-finite, $L^{\rm loc}(\Omega)$ is a complete locally convex Hausdorff space.\end{lemma}

\begin{proof}The only non-obvious assertion is that $L^{\rm loc}(\Omega)$ is complete. For all finitely measurable subsets $B,A\subset\Omega$ such that $B\subseteq A$, define the restriction maps $r_{A,B}:L^1(A)\to L^1(B):f\mapsto f|_B$. Define also the restriction maps $r_{A}:L^{\rm r-loc}(\Omega)\to L^1(A)$. In the language of \cite[\S 2.6]{ja}, the system of spaces $\{L^1(A)\}$ and maps $\{r_{A,B}\}$ forms a \emph{projective system}. Let $\mathcal{F}$ be the set of all finitely measurable subset of $\Omega$. By the definition of the projective limit (cf. \cite[\S 2.6]{ja}, $$\underset{\longleftarrow}{\lim}~L^1(A)=\{(f_A)\in\prod_{A\in\mathcal{F}}L^1(A):r_{A,B}(f_A)=f_B~\mbox{for all}~B\subseteq A\}.$$

Consider the map $m:L^{\rm loc}(\Omega)\to\underset{\longleftarrow}{\lim}~L^1(A)$ given by $f\mapsto (f|_A)$. Clearly $m$ is linear and  injective.

We show that $m$ is surjective: take any $(f_A)\in\underset{\longleftarrow}{\lim}~L^1(A)$. By the $\sigma$-finiteness of $\Omega$, there is a disjoint sequence $(W_i)$ of finitely measurable sets with $\cup W_i=\Omega$. Define $F$ of $\Omega$ by setting $F(\omega)=f_{W_i}(\omega)$ where $\omega\in W_i$. Hence, $F$ is measurable. Moreover, for any $B\subset\Omega$ of finite measure, by the Monotone Convergence Theorem, $$0=\lim_{n\to\infty}\int_{\bigcup_{i=1}^n W_i\cap B}|F-f|_B|~d\mu=\int_B|F-f|_B|~d\mu.$$ Therefore $F|_B=f_B$ a.e. Consequently, $\int_B|F|~d\mu<\infty$, $F\in L^{\rm loc}(\Omega)$ and $m(F)=(f_A)$.

Set $U_B=\{f\in L^{\rm loc}(\Omega):\int_B|f|~d\mu<1\}$ and $V_B=r^{-1}_B(\mathcal{B}(L^1(B))),$ where $\mathcal{B}(L^1(B))$ is the open unit ball in $L^1(B)$. As $m(U_B)=V_B$, and as the collections  $\{U_B\}$ and $\{V_B\}$ as $B$ ranges over all finitely measure subsets of $\Omega$, generate the respective topologies of $L^{\rm loc}(\Omega)$ and $\underset{\longleftarrow}{\lim}~L^1(A)$, we conclude that $m$ is bicontinuous. \end{proof}

We now turn to some tensor constructions of functions that will be necessary when working with the transfer operator.
 Recall first that given a $\mu$-a.e.\@ finite measurable function $f$ on $(\Omega,\mu)$, the \textit{distribution function} $s\mapsto m(f,s)$ is defined by $$m(f,s)=\mu(\{\omega\in\Omega:|f(\omega)|>s\})$$ for all $s\geq 0$. Two measurable functions $f$ and $g$ are \textit{equimeasurable} if we have $m(f,s)=m(g,s)$ for all $s\geq 0$.

\begin{definition}\label{D:skewtensor} Given a dynamical system $(G,\alpha,\Omega,\mu)$ as in Definition \ref{D:dynsys}, let $f$ and $g$ be measurable functions on $G$ and $\Omega$ respectively. The \textbf{$\boldsymbol{\alpha}$-skew tensor} $g\otimes_{\alpha}f$ is a measurable function on $\Omega\times G$ defined by $$(g\otimes_{\alpha}f)(\omega,t)=f(t)g(\alpha_t(\omega)).$$\end{definition}

There is a strong link between the skew tensor product and the standard tensor product of two functions that will come in handy.

\begin{lemma}\label{L:skewequi} Given $f$ and $g$ as above, the functions $g\otimes_{\alpha}f$ and $g\otimes f$ on $\Omega\times G$ are equimeasurable.\end{lemma}

\begin{proof} Let $\lambda\in\rn^+$ be fixed and define the following sets:\begin{eqnarray*}E&=&\{(\omega,t)\in \Omega\times G: |g\otimes_{\alpha}f(\omega,t)|>\lambda\}\\E'&=&\{(\omega,t)\in \Omega\times G: |g\otimes f(\omega,t)|>\lambda\}.\end{eqnarray*} Moreover, for a fixed $t\in G$, we define\begin{eqnarray*}E_t&=&\{\omega\in \Omega: |g(\alpha_t(\omega))|>\lambda/|f(t)|\}\\E'_t&=&\{\omega\in \Omega: |g(\omega)|>\lambda/|f(t)|\}.\end{eqnarray*}

Now because $\alpha_t(g)$ and $g$ are equimeasurable, $$\mu(E_t)=m(\alpha_t(g),\lambda/|f(t)|)=m(g,\lambda/|f(t)|)=\mu(E'_t).$$ Furthermore, $\displaystyle \mu\times h(E)=\int_G\mu(E_t)dt=\int_G\mu(E'_t)dt=\mu\times h(E').$ Hence $$m(g\otimes_{\alpha}f,\lambda)=m(g\otimes f,\lambda).$$\end{proof}

 If $f$ is a measurable function on $\Omega$, we define \begin{equation}\label{E:otimesG}F:=\otimes_{\alpha,G}(f):=f\otimes_{\alpha}\chi_G.\end{equation} In other words, $F(\omega,t)=f(\alpha_t(\omega))$. This function is measurable  on $\Omega\times G$. To see this, recall from Definition \ref{D:dynsys} that $\tilde{\alpha}:\Omega\times G\to\Omega:\omega\times t\mapsto\alpha_t(\omega)$ is measurable, which implies that $F=f\circ\tilde{\alpha}$ is measurable too. 

For the Banach space $L^1(\Omega)+L^\infty(\Omega)$ equipped with the norm $$\|h\|_{L^{1+\infty}(\Omega)}=\inf\{\|f\|_{L^1}+\|g\|_{L^\infty}:f\in L^1(\Omega),~g\in L^\infty(\Omega),~h=f+g\},$$
we will use the compressed notation $L^{1+\infty}(\Omega)$. 

If we have two measure spaces $(\Omega_1,\mu_1)$ and $(\Omega_2,\mu_2)$ then we can consider the family of seminorms $p_{A\times B}$ defined on the set of measurable functions on $(\Omega_1\times\Omega_2,\mu_1\times\mu_2)$ by setting $p_{A\times B}(f)=\int_{A\times B}|f|~d\mu_1\times\mu_2$ and define $L^{\rm r-loc}(\Omega_1\times\Omega_2)$, the space of all \textit{rectangular locally integrable functions}, to consist of those measurable functions for which all such seminorms are finite. As with the locally integrable functions, we use the family $\{p_{A\times B}\}$ to define the topology on $L^{\rm r-loc}(\Omega_1\times\Omega_2)$.  

Now $L^{\rm r-loc}(\Omega_1\times \Omega_2)$ is complete in this topology. To see this note that this space is the reduced projective limit of the system $\{L^1(A\times B)\}$ as $A$ and $B$ range over all subsets of finite measure of $\Omega_1$ and $\Omega_2$ respectively. The proof of completeness is essentially the one given in Lemma \ref{L:projLim}.

\begin{lemma}\label{L:F}If $f\in L^{1+\infty}(\Omega)$, then $F$ is rectangular-locally integrable on $\Omega\times G$. Furthermore $\otimes_{\alpha,G}$ is a continuous mapping from $L^{1+\infty}(\Omega)$ to $L^{\rm r-loc}(\Omega\times G)$.\end{lemma}

\begin{proof} Let us write $f=g_1+g_2$, where $g_1\in L^1(\Omega)$ and $g_2\in L^{\infty}(\Omega)$. Now for any subsets $K\subset G$ and $A\subset\Omega$ of finite measure, we must show that $\displaystyle\int_{ A\times K}|F|~d\mu\times h$ is finite, where as per our convention, $h$ denotes the right Haar measure on $G$.

Note that for any $t\in G$, the measure-invariance of $\alpha$ ensures that \begin{eqnarray*}\int_A |f|(\alpha_t(\omega))~d\mu(\omega)&=&\int_{\alpha_{t^{-1}}(A)}|f|(\omega)~d\mu(\omega)\\&\leq 
&\int_{\alpha_{t^{-1}}(A)}|g_1|(\omega)~d\mu(\omega)+\int_{\alpha_{t^{-1}}(A)}|g_2|(\omega)~d\mu(\omega)\\&\leq &\|g_1\|_1+|A|\|g_2\|_{\infty}.
\end{eqnarray*} By Fubini's theorem and the measurability of $F$, \begin{eqnarray*}\int_{A\times K}|F|~d\mu\times h&=&\int_K\int_A |f(\alpha_t(\omega))|~d\mu(\omega)dh(t)\\&\leq &\int_K\|g_1\|_1+|A|\|g_2\|_{\infty}~dh\\&=&|K|(\|g_1\|_1+|A|\|g_2\|_{\infty})<\infty.\end{eqnarray*} Hence $F$ is rectangular locally integrable. 

Next let $f, g$ be elements of $L^{1+\infty}(\Omega)$ which are equal $\mu$-a.e. The functions $f\otimes 1$ and $g\otimes 1$ are then trivially also equal $(\mu\times h)$-a.e. But then by Lemma \ref{L:skewequi}, the function $\otimes_{\alpha,G}(f-g)$ is equimeasurable to the 0-function, and hence equal to 0 $(\mu\times h)$-a.e. It follows that the embedding $\otimes_{\alpha,G}:L^{1+\infty}(\Omega) \to L^{\rm r-loc}(\Omega\times G)$ is well-defined. 

It remains to prove continuity. Now as $$|K|(\|g_1\|_1+|A|\|g_2\|_{\infty})<|K|(1+|A|)(\|g_1\|_1+\|g_2\|_{\infty}),$$ we have $$\displaystyle\int_{A\times K}|F|~d\mu\times h\leq |K|(1+|A|)\|f\|_{1+\infty},$$ which implies the continuity of $\otimes_{\alpha,G}$.\end{proof}

We denote by $L^{\rm k-loc}(G)$ the space of measurable functions on $G$ for which the seminorms $p_K(f)=\int_K|f|(t)\,dt$, with $K$ ranging over all the compact subsets of $G$, are all finite.

We then also define $LK^{\rm r-loc}(\Omega\times G)$ to be the space of all measurable functions on $\Omega\times G$ for which each seminorm $f\mapsto\int_{A\times K}|f|d\mu\times h$ is finite for arbitrary $A\subset \Omega$ is finitely measurable and $K\subset G$ compact.

\begin{proposition}
The spaces $L^{\rm k-loc}(G)$ and  $LK^{\rm r-loc}(\Omega\times G)$ are complete Hausdorff locally convex spaces.
\end{proposition}

\begin{proof} The proofs of all the claims except the ones regarding completeness and Hausdorffness follow fairly standard paths. We first verify the claim regarding Hausdorfness for $L^{\rm k-loc}(G)$; the proof of Hausdorffness of the topology on $LK^{\rm r-loc}(\Omega\times G)$ uses the same ideas. Note that the Haar measure is Radon, and it is easily shown that Radon measures are regular on $\sigma$-finite spaces. Given $f\neq g\in L^{\rm k-loc}(G)$,  for some $\epsilon>0$ the set $A_\epsilon=\{t:|f(t)-g(t)|>\epsilon\}$ has positive measure. The regularity of the Haar measure ensures that there is a compact subset $K$ of $A_\epsilon$ of positive measure. Consequently, $p_K(f-g)>0$. 

The proof of completeness follows from observing that the space under consideration may be written as the reduced projective limit of the system $\{L^1(K)\}$ as $K$ ranges over all compact subsets of $G$. We now follow the same argument as is Lemma \ref{L:projLim}.\end{proof}

Given a $\sigma$-finite measure space $(\Omega,\mu)$ and a Banach space $E$, it is a by now well known classical fact that the projective tensor product $L^1(\Omega)\widehat{\otimes}_\pi E$ canonically corresponds to the space $L^1(\Omega, E)$ of all $E$-valued Bochner-integrable functions \cite[15.7.5]{ja}. Equally well-known is the fact that in the particular case where $E=L^1(\Omega_2)$, one has that $L^1(\Omega)\widehat{\otimes}_\pi L^1(\Omega_2)$ is a copy of $L^1(\Omega\times\Omega_2)$ \cite{ry}. In our development we shall need versions of both these results for $L^{\rm loc}$ spaces. Since these results are of independent interest, we shall prove slightly more than we will actually need. With this in mind, we next introduce a space which may be regarded as the space of vector-valued locally integrable functions for functions on some $\sigma$-finite measure space taking values not in a Banach space, but more generally in a complete locally convex space. Let $(\Omega,\mu)$ be a $\sigma$-finite measure space and $\mathcal{E}$ a complete locally convex vector space whose topology is defined by the family $\{p_{\alpha}\}_{\alpha\in\Lambda}$ of seminorms. Here the basic underlying ideas closely parallel Jarchow's treatment of what may be regarded as integrable $\mathcal{E}$-valued functions in \cite[\S 15.7]{ja}. A $\mu$-\textit{simple measurable function} $f:\Omega\rightarrow \mathcal{E}$ is a function $f=\sum_{i=1}^N\chi_{E_i}x_i$, where $E_1,\ldots,E_N$ are $\mu$-measurable subsets of $\Omega$ and $x_1,\ldots,x_N\in \mathcal{E}$. We will write $\mathcal{S}(\Omega)$ for the space of scalar valued measurable simple functions on $\Omega$ and $\mathcal{S}(\Omega,\mathcal{E})$ for the space of all $\mathcal{E}$-valued measurable simple functions. We define the integral of such a simple measurable function $f=\sum_{i=1}^N\chi_{E_i}x_i$ over a set $A\subset\Omega$ of finite measure, by first arranging matters so that the sets $E_1,\ldots,E_N$ are mutually disjoint, and then setting $$\int_Af~d\mu=\sum_{i=1}^n\mu(A\cap E_i)x_i.$$
On this space of $\mathcal{E}$-valued simple measurable functions, we can define a family of seminorms $p_{A,\alpha}(f):=\int_Ap_\alpha(f(\omega))~d\mu(\omega)$ where $A$ ranges over all subsets of $\Omega$ of finite measure. It is an exercise to see that $p_{A,\alpha}(f)<\infty$ for each 
$f\in \mathcal{S}(\Omega,\mathcal{E})$. As in \cite[\S 15.7]{ja}, one may show that the topology induced on $\mathcal{S}(\Omega,\mathcal{E})$ by these seminorms, is a Hausdorff locally convex topology. The completion of $\mathcal{S}(\Omega,\mathcal{E})$ under this topology will be denoted by $\mathfrak{L}^{\rm loc}(\Omega,\mathcal{E})$. In closing this discussion we wish to point out that there are examples of complete locally convex spaces $\mathcal{E}$ for which not every element of $\mathfrak{L}^{\rm loc}(\Omega,\mathcal{E})$ can be written as an $\mathcal{E}$-valued function (see K\"othe's comment on p 200 of \cite{kot2}). However as we shall see in Proposition \ref{P:concrete}, for the specific case of $\mathcal{E}=L^{\rm loc}(G)$, the situation is a lot less pathological.  

We now come to the promised analogue of the classical result regarding Bochner-integrable functions.

\begin{proposition}\label{L:BochnerProj} There is a naturally defined linear bijective homeomorphism $\iota_2:\mathfrak{L}^{\rm loc}(\Omega,\mathcal{E})\to L^{\rm loc}(\Omega)\widehat{\otimes}_{\pi}\mathcal{E}$.\end{proposition}

\begin{proof}
As in the preceding discussion, let $\{p_{\alpha}\}_{\alpha\in\Lambda}$ be the family of seminorms determining the topology on $\mathcal{E}$. Let $f=\sum_{i=1}^N\chi_{E_i}x_i$ be a simple function in $\mathcal{S}(\Omega,\mathcal{E})$. Such a simple function may formally be indentified with the element  $F=\sum_{i=1}^N\chi_{E_i}\otimes x_i$ of the algebraic tensor product $\mathcal{S}(\Omega)\otimes\mathcal{E}$. An easy modification of the first four lines of the proof of \cite[Theorem 15.7.1]{ja} shows that the identification $f\to F$ is a linear bijection from $\mathcal{S}(\Omega,\mathcal{E})$ to 
$\mathcal{S}(\Omega)\otimes\mathcal{E}$. By \cite[Proposition 15.1.1]{ja} the seminorms generating the projective topology on $\mathcal{S}(\Omega)\otimes\mathcal{E}$, are of the form $$\pi_{A,\alpha}(F)=\inf\sum_{i=1}^Np_\alpha(x_i)\mu(A\cap E_i)$$where $A\subset \Omega$ is a finitely measurable subset, and where the infimum is taken over all representations of $F$ as a linear combination of simple tensors. It is an exercise to see that in the case where $E_1,\ldots,E_N$ are mutually disjoint, we actually have that $\pi_{A,\alpha}(F)=\sum_{i=1}^Np_\alpha(x_i)\mu(A\cap E_i)$. In other words we have that $\pi_{A,\alpha}(F)=p_{A,\alpha}(f)$. This clearly shows that $\mathcal{S}(\Omega,\mathcal{E})$ equipped with the topology generated by the seminorms $\{p_{A,\alpha}\}$, is linearly homeomorphic to $\mathcal{S}(\Omega)\otimes\mathcal{E}$ equipped with the projective tensor topology. Finally note that since $(\Omega,\mu)$ is assumed to be $\sigma$-finite, $\mathcal{S}(\Omega)$ can be shown to be dense in $L^{\rm loc}(\Omega)$. It therefore follows from \cite[Corollary 15.2.4]{ja} and the comment thereafter, that $\mathfrak{L}^{\rm loc}(\Omega,\mathcal{E})$ (the completion of $\mathcal{S}(\Omega,\mathcal{E})$), is homeomorphic to 
$\widetilde{\mathcal{S}(\Omega)}\widehat{\otimes}_{\pi}\mathcal{E} = L^{\rm loc}(\Omega)\widehat{\otimes}_{\pi}\mathcal{E}$. (Here $\widetilde{\mathcal{S}(\Omega)}$ denotes the completion of $\mathcal{S}(\Omega)$.)\end{proof}

Within the context of the above spaces one has the following important extension result.

\begin{proposition}\label{P:contMap}Let $\mathcal{E}$ and $\mathcal{F}$ be locally convex spaces, and let $T:\mathcal{E}\to \mathcal{F}$ be a transformation with the property that for some set of seminorms $\{p_\alpha\}$ determining the topology of $\mathcal{F}$, we can find a set of seminorms $\{q_\beta\}$ determining the topology of $\mathcal{E}$ and a positive scalar $c$, such that for any $\alpha$ we can find an $\alpha'$ such that $p_\alpha(T(x)-T(y))\leq c q_{\alpha'}(x-y)$ for all $x,y\in \mathcal{E}$.

For any $\sigma$-finite measure space $(\Omega,\mu)$ the mapping $\widetilde{T}: \mathcal{S}(\Omega,\mathcal{E})\to \mathcal{S}(\Omega,\mathcal{F})$ canonically defined by  $\widetilde{T}(f):=T\circ f$, extends to a continuous mapping from $\mathfrak{L}^{\rm loc}(\Omega,\mathcal{E})$ to $\mathfrak{L}^{\rm loc}(\Omega,\mathcal{F})$. (In the case of a linear map, the extension corresponds to the linear map $I\otimes T:L^{\rm loc}(\Omega)\widehat{\otimes}_\pi\mathcal{E} \to L^{\rm loc}(\Omega)\widehat{\otimes}_\pi\mathcal{F}$.)\end{proposition}

\begin{proof}
Let $\mathcal{S}(\Omega,\mathcal{E})$ and $\mathcal{S}(\Omega,\mathcal{F})$ respectively denote the spaces of all simple $\mathcal{E}$- and $\mathcal{F}$-valued functions. 

Each $g \in \mathcal{S}(\Omega,\mathcal{E})$ is of course of the form $g=\sum_{i=1}^k\chi_{B_{i}}g_{i}$, where $B_1,\ldots, B_k$ are measurable subsets of $\Omega$, and $g_1,\ldots, g_k$ elements of $\mathcal{E}$. Using this representation, it is an easy exercise to see that each $T\circ g$ belongs to $\mathcal{S}(\Omega,\mathcal{F})$. Hence $\widetilde{T}$ maps $\mathcal{S}(\Omega,\mathcal{E})$ into $\mathcal{S}(\Omega,\mathcal{F})$.

Let $f, g \in \mathcal{S}(\Omega,\mathcal{E})$ be given. Recall that by hypothesis we have that 
$$p_\alpha(T(f(\omega))-T(g(\omega)))\leq c q_{\alpha'}(f(\omega)-g(\omega)) \mbox{ for each }\omega\in \Omega.$$It then follows from these inequalities that  
\begin{eqnarray}\label{zzz} 
p_{A,\alpha}(T\circ f - T\circ g)&=&\int_A p_\alpha((T\circ f)(\omega))-(T\circ g)(\omega)))\,d\mu(\omega)\nonumber\\
&\leq& c \int_A q_{\alpha'}(f(\omega)-g(\omega))\,d\mu(\omega)\\
&=& c q_{A,\alpha'}(f-g)\nonumber
\end{eqnarray}
for every finitely measurable $A\subset\Omega$. This clearly suffices to ensure that $\widetilde{T}$ is actually continuous from $\mathcal{S}(\Omega,\mathcal{E})$ into $\mathcal{S}(\Omega,\mathcal{F})$. 

It remains to show that $\widetilde{T}$ has a unique continuous extension to all of $\mathfrak{L}^{\rm loc}(\Omega,\mathcal{E})$. Let $\{f_\gamma\}$ and $\{f_\beta\}$ be two nets in $\mathcal{S}(\Omega,\mathcal{E})$, converging to some $f\in \mathfrak{L}^{\rm loc}(\Omega,\mathcal{E})$. Using inequality \ref{zzz} one can show that both of the nets $\{T\circ f_\gamma\}$ and $\{T\circ f_\beta\}$ are also Cauchy. By the completeness of $\mathfrak{L}^{\rm loc}(\Omega,\mathcal{F})$, both these nets must converge. However since for any $\gamma$ and $\beta$ we have that 
$$p_{A,\alpha}(T\circ f_\gamma - T\circ f_\beta) \leq c q_{A,\alpha'}(f_\gamma-f_\beta) = c [q_{A,\alpha'}(f_\gamma-f) + q_{A,\alpha'}(f-f_\beta),$$
it is clear that $\{T\circ f_\gamma\}$ and $\{T\circ f_\beta\}$ must converge to the \emph{same} element of $\mathfrak{L}^{\rm loc}(\Omega,\mathcal{F})$. We therefore define the extension of $\widetilde{T}$ to all of $\mathfrak{L}^{\rm loc}(\Omega,\mathcal{E})$ by simply defining $\widetilde{T}(f)$ to be the limit of the net $\{T\circ f_\gamma\}$. 

Finally let $f,g \in \mathfrak{L}^{\rm loc}(\Omega,\mathcal{E})$ be given, and select nets $\{f_\gamma\}$ and $\{g_\beta\}$ respectively converging to $f$ and $g$. On taking limits, it is then an exercise to see that the inequality $p_{A,\alpha}((T\circ f_\gamma)-(T\circ g_\beta)) \leq c q_{A,\alpha'}(f_\gamma-g_\beta)$, ensures that we also have that $$p_{A,\alpha}(\widetilde{T}(f)-\widetilde{T}(g)) \leq c q_{A,\alpha'}(f-g).$$Using this inequality, one may then show that the extended map is in fact continuous.
\end{proof}

The next order of business is to verify the other promised analogue of the classical tensor product results. Some preparation is required for this. Let $S_r(\Omega_1\times \Omega_2)$ denote the rectangular simple functions on $\Omega_1\times \Omega_2$, that is, all functions of the form $\sum_{i=1}^n c_i\chi_{A_i\times B_i}$ where the $c_i\in\cn$ and the $A_i$ and $B_i$ are subsets of finite measure of $\Omega_1$ and $\Omega_2$ respectively. 

\begin{lemma}\label{L:rectLimit}Let $(\Omega_1,\mu_1)$ and $(\Omega_2,\mu_2)$ be two $\sigma$-finite measure spaces and $G$ a locally compact $\sigma$-group. Then $S_r(\Omega_1\times \Omega_2)$ is dense in both $L^{\rm r-loc}(\Omega_1\times\Omega_2)$ and $LK^{\rm r-loc}(\Omega_1\times G)$.\end{lemma}

\begin{proof}We prove the claim for $L^{\rm r-loc}(\Omega_1\times\Omega_2)$; the claim for $LK^{\rm r-loc}(\Omega_1\times G)$ is similar. Let $f\in L^{\rm r-loc}(\Omega_1\times\Omega_2)$ be a $[0,\infty]$-valued function. To prove the Lemma, it suffices to show that there is an increasing sequence $(f_n)$ of non-negative rectangular simple functions such that $\lim_{n\to\infty}f_n(\omega_1,\omega_2)=f(\omega_1,\omega_2)$ $\mu_1\times\mu_2$-a.e., because then by the Monotone Convergence Theorem$$\lim_{n\to\infty}\int_{A\times B}|f-f_n|~d\mu_1\times\mu_2=0$$ for any $\mu_1\times\mu_2$-finite rectangle $A\times B$. By \cite[Theorem 1.17]{ruRC} there is an increasing sequence of simple functions $(g_n)$ converging pointwise a.e.\@ to $f$. 

Let us now observe how a simple function $g=\sum_{i=1}^m\chi_{E_i}c_i$ in $L^{\rm r-loc}(\Omega_1\times\Omega_2)$ can be approximated by a rectangular simple function of finite support. Fix an $\epsilon>0$. The measurability of each $E_i$ in $\Omega_1\times\Omega_2$ implies that there is a set $E'_i\subseteq E_i$ such that $E'_i$ is the union of finitely many rectangles and $\mu_1\times\mu_2(E_i\backslash E'_i)<\epsilon/2^i$. Then $g'=\sum_{i=1}^m\chi_{E'_i}c_i$ is a rectangular simple function and $g'$ differs from $g$ on a set of measure at most $\epsilon$, and $g'\leq g$.

Starting from the sequence $(g_n)$, define $f_n$ to be the approximation of $g_n$ as described in the previous paragraph such that supp$(f_{n+1})\supseteq$supp$(f_n)$ and $$\mu_1\times\mu_2({\rm supp}(g_n)\backslash{\rm supp}(f_n))<1/n$$ for all $n\in\nn$. Clearly, $$\lim_{n\to\infty}f_n(\omega_1,\omega_2)=\lim_{n\to\infty}g_n(\omega_1,\omega_2)=f(\omega_1,\omega_2)$$ for a.e.\@ $(\omega_1,\omega_2)\in \Omega_1\times\Omega_2$, proving the Lemma.\end{proof}

\begin{proposition}\label{P:concrete}Let $\Omega_1$ and $\Omega_2$ be two $\sigma$-finite measure spaces. There is a bijective linear homeomorphism $\iota:L^{\rm loc}(\Omega_1, L^{\rm loc}(\Omega_2))\to L^{\rm r-loc}(\Omega_1\times \Omega_2)$. There is similarly a bijective linear homeomorphism $\iota_2:L^{\rm loc}(\Omega_1, L^{\rm k-loc}(G))\to LK^{\rm r-loc}(\Omega_1\times \Omega_2)$.\end{proposition}

\begin{proof}As there are only minor technical differences in the proofs of the two claims, we shall only prove the first claim.

The elements of the space $\mathcal{S}(\Omega_1,\mathcal{S}(\Omega_2))$ can be written as $\sum_{i=1}^nf_i\chi_{A_i}$, where each $f_i\in \mathcal{S}(\Omega_2)$ and the $A_i$'s are mutually disjoint measurable subsets of $\Omega_1$. Topologise $\mathcal{S}(\Omega_1,\mathcal{S}(\Omega_2))$ by the family $\{\pi_{A,B}\}$ of seminorms given by $$\pi_{A,B}(\sum_{i=1}^nf_i\chi_{A_i}):=\sum_{i=1}^n\mu_1(A_i\cap A)\int_B|f_i(t)|dt.$$

Topologise $S_r(\Omega_1\times \Omega_2)$ with the seminorms $\{p_{C,D}\}$, as $C$ and $D$  range over all subsets of finite measure of $\Omega_1$ and $\Omega_2$ respectively, and $$p_{C,D}(\sum_{i=1}^n c_i\chi_{A_i\times B_i})=\sum_{i=1}^nc_i\mu_1\times\mu_2((C\cap A_i)\times (D\cap B_i)).$$(Here we once again need to assume that the $A_i$'s and $B_i$'s are mutually disjoint.) 

We shall first show that there is a bijective linear homeomorphism \begin{equation}\label{E:concrete}\iota:\mathcal{S}(\Omega_1,\mathcal{S}(\Omega_2))\to S_r(\Omega_1\times \Omega_2).\end{equation} Then we shall show that $\mathcal{L}^{\rm loc}(\Omega_1,L^{\rm loc}(\Omega_2))$ and $L^{\rm r-loc}(\Omega_1\times \Omega_2)$ are the respective completions of $\mathcal{S}(\Omega_1,\mathcal{S}(\Omega_2))$ and $S_r(\Omega_1\times \Omega_2)$. The extension of the homeomorphism in (\ref{E:concrete}) to the completions of these two spaces, proves the result.

Define $\iota$ by setting $\iota(f)=F$, where $F(\omega_1,\omega_2)=f(\omega_1)(\omega_2)$ for all $f\in \mathcal{S}(\Omega_1,\mathcal{S}(\Omega_2))$, $\omega_1\in\Omega_1$ and $\omega_2\in \Omega_2$. It is easy to see that $F$ is a rectangular simple function and that $\iota$ is a linear bijection of the spaces. For all $f=\sum_{i=1}^nf_i\chi_{A_i}$ in $\mathcal{S}(\Omega_1,\mathcal{S}(\Omega_2))$ with the $A_i$'s mutually disjoint, the obvious equality 
\begin{eqnarray*}\pi_{A,B}(f)&=&\pi_{A,B}(\sum_{i=1}^nf_i\chi_{A_i})=\sum_{i=1}^n\mu_1(A_i\cap A)\int_B|f_i(\omega_2)|d\mu_2(\omega_2)\\&=&\int_A\int_B|F(\omega_1,\omega_2)|~d\mu_2(\omega_2)d\mu_1(\omega_1)=p_{A,B}(F)\end{eqnarray*} shows that both $\iota$ and its inverse are continuous.

Now we show that the completion of $\mathcal{S}(\Omega_1,\mathcal{S}(\Omega_2))$ is $\mathfrak{L}^{\rm loc}(\Omega_1,L^{\rm loc}(\Omega_2))$. An obvious modification of the argument of Proposition \ref{L:BochnerProj} shows that $\mathcal{S}(\Omega_1,\mathcal{S}(\Omega_2))$ may be linearly identified with the algebraic tensor product $\mathcal{S}(\Omega_1)\otimes\mathcal{S}(\Omega_2)$, and that this linear map is a homeomorphism when $\mathcal{S}(\Omega_1,\mathcal{S}(\Omega_2))$ is equipped with the locally cited topology and $\mathcal{S}(\Omega_1)\otimes\mathcal{S}(\Omega_2)$ with the projective tensor topology. If now we take completions, then by Proposition \ref{L:BochnerProj}, the fact that $\widetilde{\mathcal{S}(\Omega_1)}\widehat{\otimes}_{\pi}\widetilde{\mathcal{S}(\Omega_1)} = L^{\rm loc}(\Omega_1)\widehat{\otimes}_{\pi}L^{\rm loc}(\Omega_1)$ (see \cite[Corollary 15.2.4]{ja}), translates to the claim that the completion of $\mathcal{S}(\Omega_1,\mathcal{S}(\Omega_2))$ is $\mathfrak{L}^{\rm loc}(\Omega_1,L^{\rm loc}(\Omega_2))$. 

Finally, by Lemma \ref{L:rectLimit}, $S_r(\Omega_1\times \Omega_2)$ is dense in $L^{\rm r-loc}(\Omega_1\times \Omega_2)$.
\end{proof}

One of the very useful consequences of the above result, is that the space $\mathfrak{L}^{\rm loc}(\Omega,C(G))$ may be regarded as a space of measurable functions on $(\Omega\times G)$. This follows from the next lemma.

\begin{corr}\label{cgconcrete} The space $\mathfrak{L}^{\rm loc}(\Omega,C(G))$ continuously and canonically injects into $LK^{\rm r-loc}(\Omega\times G).$\end{corr}

\begin{proof} Let $A$ be a finitely measurable subset of $\Omega$ and $K$ a compact subset of $G$. It is an easy exercise to see that each of the canonical mappings $C(G)\to L^{\rm k-loc}(G):f\to f$, $C(K)\to L^1(K):f\to f$, $L^{\rm loc}(\Omega)\to L^1(A):f\to f|_A$, $L^{\rm k-loc}(G)\to L^1(K):f\to f_K$ and $C(G)\to C(K):f\to f|_K$ is continuous. On combining \cite[Proposition 15.2.1\& Corollary 15.7.5]{ja} and Proposition \ref{L:BochnerProj}, it follows that these maps induce canonical maps $\mathfrak{L}^{\rm loc}(\Omega,C(G))\to \mathfrak{L}^{\rm loc}(\Omega,L^{\rm k-loc}(G))$, $\mathfrak{L}^{\rm loc}(\Omega,C(G)) \to L^1(A,C(K))$ and 
$L^1(A,C(K))\to L^1(A,L^1(K))$. Now consider the diagram below:

\begin{equation*}\begin{tikzpicture}[.=stealth',xscale=3.5,yscale=2]
\node (a) at (0,1){$\mathfrak{L}^{\rm loc}(\Omega,C(G))$};
\node (b) at (0,0){$L^1(A,C(K))$};
\node (c) at (1,1){$\mathfrak{L}^{\rm loc}(\Omega,L^{\rm k-loc}(G))$};
\node (d) at (1,0){$L^1(A,L^1(K))$};

\draw[->] (a) edge node[above]{$\iota_1$} (c);
\draw[->] (a) edge node[left]{$\iota_3$} (b);
\draw[->] (b) edge node[below]{$\iota_4$} (d);
\draw[->] (c) edge node[right]{$\iota_2$} (d);
\end{tikzpicture}\end{equation*}
By first proving this to be true for elements of $\mathcal{S}(\Omega,C(K))$ and then extending by continuity, one may show that this diagram commutes. 

Now suppose we are given some $f\in \mathfrak{L}^{\rm loc}(\Omega,C(G))$ for which $\iota_1(f)=0$. To establish the Corollary, we need to show that this can only true for all pairs $(A,K)$ if $f=0$. 

Observe that since $\iota_1(f)=0$, we have that $\iota_4\circ\iota_3(f)=\iota_2\circ\iota_1(f)=0$. Next note that both of the spaces $L^1(A,C(K))$ and $L^1(A,L^1(K))$ may canonically be written as functions on $(A\times K)$. Using this fact, it is an interesting exercise to show that the canonically defined embedding $\iota_4:L^1(A,C(K))\to L^1(A,L^1(K))$ is injective. Hence we must have that $\iota_3(f)=0$. But $\mathfrak{L}^{\rm loc}(\Omega,C(G))$ can be shown to be the projective limit of the spaces $L^1(A,C(K))$ with the map $\iota_3$ merely being the projection from $\mathfrak{L}^{\rm loc}(\Omega,C(G))$ to $L^1(A,C(K))$. Thus the fact that $\iota_3(f)=0$ for every pair $(A,K)$, guarantees that $f=0$, as required.

(In closing we briefly justify the claim regarding projective limits. Let $K$ be compact subset of $G$, and $p_K$ the associated seminorm on $C(G)$. Since every locally compact group is a normal topological space \cite{fo}, we have access to the Tietze extension theorem in analysing $C(G)$. In particular $C(K)$ may be written as $C(K)\equiv\{f|_K:f\in C(G)$. This fact enables us to write $C(G)$ as the reduced projective limit of the family of spaces $\{C(K)\}$ and mappings $\{r_{i,j}:C(K_i)\to C(K_j)\}$ where $K_i\supseteq K_j$ and $r_{i,j}$ is the restriction map, with $K$ ranging over all compact subsets $\Omega$. The space $L^{\rm loc}(\Omega)$ may similarly be written as the reduced projective limit of the spaces $L^1(A)$ where $A$ ranges over all finitely measurable subsets of $\Omega$, and where the intertwining maps are once again appropriate restriction maps. Hence by \cite[Theorem 15.4.2 \& Corollary 15.7.5]{ja}, $\mathfrak{L}^{\rm loc}(\Omega,C(G))\equiv L^{\rm loc}(\Omega)\widehat{\otimes}_{\pi}C(G)$ is the reduced projective limit of the family of spaces $L^1(A,C(K))\equiv L^1(A)\widehat{\otimes}_{\pi}C(K)$ of $C(K)$-valued Bochner-integrable functions, as $A$ ranges over all finitely measurable subsets of $\Omega$ and $K$ over all compact subsets of $G$ respectively.)
\end{proof}

We are now finally ready to construct the transfer operator $T^\#$.

\begin{definition}\label{R:transfer} Let $T:L^{\rm loc}(G)\to C(G)$ be a quasi-bounded operator which is either linear, or positive-valued sublinear. The transferred operator $T^\#: L^{1+\infty}(\Omega)\to L^{\rm loc}(G)$ $T^\#(f)$ is defined as the composition of the maps in the following diagram:
\begin{equation}\label{E:sublinmettrans}\begin{tikzpicture}[xscale=2.5,yscale=0.8,>=stealth']
\node (a) at (0,1){$L^{1+\infty}(\Omega)$};
\node (b) at (1,-1){$L^{\rm r-loc}(\Omega\times G)$};
\node (c) at (2,1){$\mathfrak{L}^{\rm loc}(\Omega,L^{\rm loc}(G))$};
\node (d) at(3,-1){$\mathfrak{L}^{\rm loc}(\Omega,C(G))$};
\node (e) at (4,1){$L^{\rm loc}(\Omega).$};

\draw[->] (a) edge node[above right]{$\otimes_{\alpha,G}$} (b);
\draw[->] (b) edge node[above]{$\iota^{-1}$} (c);
\draw[->] (c) edge node[above]{$\widetilde{T}$} (d);
\draw[->] (d) edge node[above]{$\widetilde{\epsilon}_1$} (e);

\end{tikzpicture}\end{equation}

Here $\widetilde{T}$ is defined as in Proposition \ref{P:contMap} by the formula $\widetilde{T}(f)=T\circ f$ for all $f\in \mathfrak{L}^{\rm loc}(\Omega,L^{\rm loc}(G))$, and $\iota^{-1}$ is the inverse of the homeomorphism of Proposition \ref{P:concrete}. From the point evaluation $\epsilon_1:C(G)\to\cn$ at $1\in G$, we derive the map $\widetilde{\epsilon}_1:\mathcal{L}^{\rm loc}(\Omega,C(G))\to\mathcal{L}^{\rm loc}(\Omega,\cn)\simeq L^{\rm loc}(\Omega)$ using Proposition \ref{P:contMap} again. (The homeomorphism $\mathcal{L}^{\rm loc}(\Omega,\cn)\simeq L^{\rm loc}(\Omega)$ is a direct consequence of Proposition \ref{L:BochnerProj}.)
\end{definition}

One question remains, and is the extent to which the above definition of the transfer operator corresponds to Calder\'ons original conception of the transfer operator. Although it may be somewhat idealistic to expect the two definitions to always fully agree, as is shown by the next proposition, we do have agreement in the cases where it matters most. This result will be extensively and silently used throughout Section \ref{S:weakTypeDef}.

\begin{proposition}\label{P:concreteRealisation} For any $F\in L^{\rm r-loc}(\Omega,G)$ and any finitely measurable subset $E\subset G$, we have that $\widetilde{T}(F_E)(\omega,t)= T((F_E)_\omega)(t)$ for almost every $\omega \in \Omega$ and $t\in G$ where $F_E:=f\otimes_{\alpha}\chi_E$ on $\Omega\times G$. (Formally $F_E=\chi_E F$.)
\end{proposition}

\begin{proof} Let $E$ be as in the hypothesis, and $\{A_n\}$ and $\{K_n\}$ respectively be a sequence of disjoint finitely measurable subsets of $\Omega$ such that $\Omega=\cup_{n=1}^\infty A_n$,and a sequence of disjoint compact subsets of $G$ such that $G=\cup_{n=1}^\infty K_n$. Fix any $n$ and $m$. 

The space $L^1(E)$ can of course be canonically embedded in $L^{\rm loc}(G)$ by simply extending the action of the former space by assigning the value 0 to all those functions at every $t\in E^c$. When embedded in this way, it is moreover an exercise to see that for nets in $L^1(E)$ convergence in $L^1(E)$ is the same as convergence in the topology inherited from $L^{\rm loc}(G)$. Thus the restriction of $T$ induces a quasi-bounded map from $L^1(E)$ into $C(G)$. 

The multiplier map induced by $\chi_{K_m}$ maps $C(G)$ onto $C(K_m)$, which in turn embeds into $L^1(K_m)$. Thus the prescription $f\to \chi_{K_m}T(f)$ therefore yields a quasi-bounded operator from $L^1(E)$ into $L^1(K_m)$. It is now an exercise to use this quasi-boundedness to see that for any $n$, the operator $S_T: L^1(A_n,L^1(E))\to L^1(A_n,L^1(K_m))$ given by $S_T(h)(\omega)\chi_{K_m}T(h(\omega))$ for any $h \in L^1(A_n,L^1(E))$, is in fact continuous. Recalling that $L^1(A_n,L^1(E))\equiv L^1(A_n)\widehat{\otimes}_\pi L^1(E) \equiv L^1(A_n\times E)$ with a similar conclusion holding for $L^1(A_n,L^1(K_m))$, it is clear that we may equivalently regard $S_T$ as a map from $L^1(A_n\times E)$ to $L^1(A_n\times K_m)$.

With $M_{\chi_E}$ and $M_{\chi_{A_n}}$ denoting the appropriate multiplier maps on $L^{\rm loc}(\Omega)$ and $L^{\rm loc}(G)$, it is easy to see that the tensor operator $M_{\chi_E}\otimes M_{\chi_E}$ continuously maps $L^{\rm loc}(\Omega)\widehat{\otimes}_\pi L^{\rm loc}(G)$ onto $L^1(A_n)\widehat{\otimes}_\pi L^1(E)$. Equivalently the prescription $F\to \chi_{A_n}F_E$ defines a continuous operator from $L^{\rm r-loc}(\Omega\times G)$ onto $L^1(A_n\times E)$.

Thus for any $F\in L^{\rm r-loc}(\Omega\times G)$ the prescription $F\to S_T(\chi_{A_n}F_E)$ defines a continuous map from $L^{\rm r-loc}(\Omega\times G)$ to $L^1(A_n\times K_m)$. A careful consideration of the construction reveals that for each $F\in L^{\rm r-loc}(\Omega\times G)$, the element $S_T(\chi_{A_n}F_E)$ is defined by $S_T(\chi_{A_n}F_E)(\omega,t)= T((F_E)_\omega)(t)$ for almost every $\omega \in A_n$ and $t\in K_m$.

The space $L^{\rm r-loc}(\Omega \times E)$ may be canonically identified with the subspace of $L^{\rm r-loc}(\Omega\times G)$ consisting of all elements embedded for almost every $\omega \in \Omega$ and $t\in G$. Using a similar argument as before it therefore follows that the sequence of operations $F\to F_E \to \widetilde{T}{F_E}\to \chi_{A_n}\chi_{K_m}\widetilde{T}{F_E}$ yields a continuous map from $L^{\rm r-loc}(\Omega\times G)$ to $L^1(A_n\times K_m)$. However from the Proposition \ref{P:contMap}, it is clear that the map we have just described, must agree with the one described in the previous paragraph on $\mathcal{S}_r(\Omega, L^{\rm loc}(G))$. So by continuity, they must agree on all of $L^{\rm r-loc}(\Omega \times E)$. This amounts to the claim that $\widetilde{T}(F_E)(\omega,t)= T((F_E)_\omega)(t)$ for almost every $\omega \in A_n$ and $t\in K_m$. Since $n$ and $m$ were arbitrary, the validity of the lemma follows.
\end{proof}


\subsection{The effect of semilocality and translation invariance}\label{SS:sal}

\emph{In the sequel, all our transferable operators will be either linear or positive-valued sublinear, as well as continuous, semilocal, and translation invariant.}

We now describe some approximation properties of $T^{\#}$ that will be useful in the next section. We start by extending some constructions that we have used earlier. Let $K$ be any measurable subset of $G$. As in equation (\ref{E:otimesG}), we can define the operator $\otimes_{\alpha,K}$ on the set of measurable functions on $\Omega$ by setting $F_K:=\otimes_{\alpha,K}(f):=f\otimes_{\alpha}\chi_K$ on $\Omega\times G$ using Definition \ref{D:skewtensor}. Hence\begin{equation}\label{E:F_K}F_K(\omega,t)=\left\{\begin{array}{ccc}f(\alpha_t(\omega))&{\rm if}&t\in K\\0&{\rm if}&t\notin K.\end{array}\right.\end{equation} 

In this notation, $F_G=F$. In the sequel we shall also use the notation $F'_K=\widetilde{T}(F_K)$ where $\widetilde{T}$ is as given in  the diagram (\ref{E:sublinmettrans}) of Definition \ref{R:transfer}. In particular, $F'=F'_G$. By a slight abuse of notation, we may also view the functions $F'_K$ as measurable functions on $\Omega\times G$ (see Corollary \ref{cgconcrete}).

\begin{lemma}\label{L:Kdominance}Let $T$ be a transferable operator and let $U$ be the open neighbourhood guaranteed by Definition \ref{D:to}(2). Let $K$ and $E$ be measurable subsets of $G$ such that $EU^{-1}\subseteq K$. For any $f\in L^{1+\infty}(\Omega)$, and almost all $(\omega,t)\in \Omega\times E$, we have \begin{equation}\label{E:Kdominance}|F'(\omega,t)|\leq |F'_K(\omega,t)|.\end{equation}\end{lemma}

\begin{proof} First note that $F_{K^c}=f\otimes_{\alpha}(\chi_{G}-\chi_{K})=f\otimes_{\alpha}\chi_{G}-f\otimes_{\alpha}\chi_{K}=F-F_{K}.$ Consequently,

\begin{eqnarray*}|F'|=|\widetilde{T}(F)|&=&|\widetilde{T}(F-F_{K}+F_{K})|\\&\leq&|\widetilde{T}(F_{K^c})|+|\widetilde{T}(F_{K})|\\&=&|F'_{K^c}|+|F'_K|.\end{eqnarray*}

We proceed to show that $|F'_{K^c}|=0$. For this we first assume that $F \in \mathcal{S}(\Omega,L^{\mathrm{loc}}(G))$. It is an exercise to see that then also $F_{K^c} \in \mathcal{S}(\Omega,L^{\mathrm{loc}}(G))$. Hence in this case $\widetilde{T}(F_{K^c})$ is defined by by $T\circ F(\omega,.)$ for $\mu$-ae $\omega$. Because $EU^{-1}\subseteq K$, it follows that $(K^cU)\cap E$ is empty since $a\in (K^cU)\cap E$ implies that there is a $b\in U$ so that $ab^{-1}\in K^c\cap(EU^{-1})$, which is impossible if $EU^{-1}\subseteq K$. Thus in this case the semilocality of $T$, ensures that the measurable map $t\mapsto F'_{K^c}(\omega,t)$ has support in $K^cU$ for almost every $\omega\in\Omega$. Hence $|F'_{K^c}|=0$. 

For the general case select a net $\{F_\gamma\}\subset \mathcal{S}(\Omega,L^{\mathrm{loc}}(G))$ converging to $F$ in $\mathfrak{L}^{\mathrm{loc}}(\Omega, L^{\mathrm{loc}}(G))$. The map $f\to f\chi_{K^c}$ is easily seen to be continuous on $L^{\mathrm{loc}}(G)$. By using the representation in Proposition \ref{L:BochnerProj}, one can show that this map canonically extends to a continuous map on $\mathfrak{L}^{\mathrm{loc}}(\Omega, L^{\mathrm{loc}}(G))$. It follows that $\{(F_\gamma)|_{K^c}\}$ converges to $F_{K^c}$. Hence by the continuity of $\widetilde{T}$, $\{(F_\gamma)'|_{K^c}\}$ must converge to $F'_{K^c}$. Since $(F_\gamma)'|_{K^c}=0$ for every $\gamma$, we must have that $F'_{K^c}=0$ as required.\end{proof}

Translation invariant operators, a class that includes all convolution operators, are automatically equi\-measura\-bility-preserving, in a sense made precise by the following Lemma.

\begin{lemma}\label{L:equi}Let $T$ be a transferable operator. For any $f\in L^{1+\infty}(\Omega)$, all $s,t\in G$ and almost all $\omega\in\Omega$,$$F'(\alpha_s(\omega),t)=F'(\omega,ts).$$ Moreover, for any $t_1,t_2\in G$, the mappings $\omega\mapsto F'(\omega,t_1)$ and $\omega\mapsto F'(\omega,t_2)$ are equimeasurable.\end{lemma}

\begin{proof} Let $T$ be a sublinear transferable operator. One may conclude from the measure preserving action of the $\alpha_s$'s, that each such transformation induces a well-behaved endomorphism $\widetilde{\alpha_s}$ on both $\mathfrak{L}^{\mathrm{loc}}(\Omega, L^{\mathrm{loc}}(G))$ and $\mathfrak{L}^{\mathrm{loc}}(\Omega, C(G))$. (For the sake of simplicity we use the same notation for both of these classes of endomorphisms.) These are defined in the obvious way by for example setting $\widetilde{\alpha_s}(F)(\omega,t)= G(\alpha_s(\omega),t)$ for each $F\in \mathcal{S}(\Omega, L^{\mathrm{loc}}(G))$, then checking continuity on $\mathcal{S}(\Omega, L^{\mathrm{loc}}(G))$, and finally extending to all of $\mathfrak{L}^{\mathrm{loc}}(\Omega, L^{\mathrm{loc}}(G))$ by continuity. Moreover in their action on $L^{\mathrm{loc}}(\Omega, L^{\mathrm{loc}}(G))$, the operators $\widetilde{T}\circ \widetilde{\alpha_s}$ and $\widetilde{\alpha_s}\circ\widetilde{T}$ agree. To see this notice that for $F\in \mathcal{S}(\Omega, L^{\mathrm{loc}}(G))$, we have
\begin{eqnarray*}
(\widetilde{\alpha_s}\circ\widetilde{T})(F)(\omega,t)&=&F'(\alpha_s(\omega),t)\\
&=&(T(F))_{\alpha_s(\omega)}(t)\\
&=&T(F_{\alpha_s(\omega)})(t)\\
&=&(\widetilde{T}\circ\widetilde{\alpha_s})(F)(\omega,t).
\end{eqnarray*}
The conclusion follows on extending by continuity. The right-translation operators $\tau_s$ defined in Definition \ref{D:to}, may similarly be extended to endomorphisms $\widetilde{\tau_s}$ on  both of $\mathfrak{L}^{\mathrm{loc}}(\Omega, L^{\mathrm{loc}}(G))$ and $\mathfrak{L}^{\mathrm{loc}}(\Omega, C(G))$. By way of example one does this for $\mathfrak{L}^{\mathrm{loc}}(\Omega, L^{\mathrm{loc}}(G))$ by setting $\widetilde{\tau_s}(F)(\omega,t)=F(\omega,ts)$ for each $F\in\mathcal{S}(\Omega, L^{\mathrm{loc}}(G))$, and then checking continuity on $\mathcal{S}(\Omega, L^{\mathrm{loc}}(G))$, before extending the action by continuity. It is a simple matter to check that the commutation of $T$ and $\tau_s$ in Definition \ref{D:to} translates to the commutation of $\widetilde{T}$ and $\widetilde{\tau_s}$. (As before one checks this for elements of $\mathcal{S}(\Omega, L^{\mathrm{loc}}(G))$,and then extend by continuity.)

Now suppose that $F$ is as in the hypothesis of the Lemma, namely of the form $F(\omega,t)=f(\alpha_t(\omega)$ for some $f\in L^{1+\infty}(\Omega)$. For  $\mu$-a.e.\@ $\omega$ and any $s,t\in G$, we then have $$[\widetilde{\alpha_s}(F)](\omega,t)=F(\alpha_s(\omega),t)=f\circ\alpha_t(\alpha_s(\omega))=f\circ\alpha_{ts}(\omega)=F(\omega,ts)=[\widetilde{\tau_s}(F)(\omega,t)],$$

For such $F$ it therefore follows that
\begin{eqnarray*}
F'(\alpha_s(\omega),t) &=&(\widetilde{\alpha_s}\circ\widetilde{T})(F)(\omega,t)\\
&=&(\widetilde{T}\circ\widetilde{\alpha_s})(F)(\omega,t)\\
&=&(\widetilde{T}\circ\widetilde{\tau_s})(F)(\omega,t)\\
&=&[(\widetilde{\tau_s}\circ\widetilde{T})F](t,\omega)\\
&=& \widetilde{\tau_s}F'(t,\omega)\\
&=& F'(ts,\omega).
\end{eqnarray*}

Finally, let $s=t^{-1}_1t_2$ and $\lambda>0$. Then as $F'(\omega,t_2)=F'(\alpha_s(\omega),t_1)$, we see that \begin{eqnarray*}\mu(\{\omega:|F'(\omega,t_2)|>\lambda\})&=&\mu(\{\omega:|F'(\alpha_s(\omega),t_1)|>\lambda\})\\&=&\mu(\{\alpha_s^{-1}(\omega):|F'(\omega,t_1)|>\lambda\})\\&=&\mu(\{\omega:|F'(\omega,t_1)|>\lambda\}),\end{eqnarray*} proving the equimeasurability of the maps $\omega\mapsto |F'(\omega,t_1)|$ and $\omega\mapsto |F'(\omega,t_2)|$.\end{proof}

The last two lemmas will both be needed in the proof of Lemma \ref{L:preCalderonTheorem1}.

\subsection{Examples}

One of the main sources of transferable operators in applications is convolution operators. We prove a Lemma that shows how such operators fit neatly into the scheme that we have developed.

\begin{lemma}\label{L:convWorks}Let $G$ be a $\sigma$-finite locally compact group and let $k\in L^1(G)$ have finite support. Then the convolution operator $T(f):=k*f$ is a continuous, semilocal, and translation invariant operator from $L^{\rm loc}(G)$ to $C(G)$.\end{lemma}

\begin{proof}That $T$ is semilocal and translation invariant follows directly from the properties of convolution. 

Let $S$ be the closure of the support of $k$. By hypothesis $S$ has finite measure. Recall that the topology on $C(G)$ is generated by the seminorms $q_K(f):=\sup\{|f(t)|:t\in K\}$, where $K$ ranges over the compact subsets of $G$, and the topology on $L^{\rm loc}(G)$ by the seminorms $p_A(f)= \int_A|f|(s)~ds$, where $A$ ranges over all finitely measurable subsets of $G$.

For any $f\in L^{\rm loc}(G)$, we now have the following straightforward computation:\begin{eqnarray*}q_K(T(f))&=&\sup_{t\in K}|k*f(t)|
\\&=&\sup_{t\in K}\left|\int_S k(s)f(ts^{-1})~ds\right|\\&=&\sup_{t\in K}\left|\int_S k(s)f|_{KS^{-1}}(ts^{-1})~ds\right|\\
&\leq&\|k\|_{L^1}\|f|_{KS^{-1}}\|_{L^1}\\
&=&\|k\|_{L^1}p_{KS^{-1}}(f).\end{eqnarray*}
Since $KS^{-1}$ is finitely measurable, this inequality ensures the continuity of $T$ at 0, and hence on all of $C(G)$.
\end{proof}

The straightforward construction of ergodic averaging operators from convolution operators demonstrates the utility and ubiquity of the transfer operator construction. 

Suppose that $\mathcal{O}$ is a measure-preserving automorphism on the measure space $(\Omega,\mu)$. It induces an action $\alpha$  of $\zn$ on $\Omega$ via $\alpha(n):=\mathcal{O}^n$. If $S$ is a finite subset of $\zn$, let $T_S$ be the convolution operator defined on the space of all locally integrable functions $f$ on $\zn$ by $$T_S(f):=\frac{1}{|S|}\chi_{-S}*f,$$ where of course $\chi_{-S}$ is the characteristic function of $-S$. Bearing in mind that the set of locally integrable functions on $\zn$ is precisely the set of all complex-valued functions, we see that $T_S$ is well-defined, linear and takes its values in $C(\zn)$, which is metrisable. From the properties of convolution, it is clearly semi-local. Indeed, if $N=\max\{|s|:s\in S\}$ and $f$ is a function with support in $[-M,M]$, then $T_S(f)$ will have support in $[-M-N,M+N]$.

Let us now determine the transfer operator $T_S^{\#}$. Let $f\in L^{1+\infty}(\Omega)$. By Definition \ref{R:transfer},
\begin{eqnarray*}\widetilde{T_S}(F(t,\omega))&=&\widetilde{T_S}(f(\alpha_t(\omega)))
\\&=&\frac{1}{|S|}\sum_{s\in S}f(\alpha_{t+s}(\omega)).\end{eqnarray*}
Hence $$T_S^{\#}(f)(\omega)=\widetilde{\epsilon}_0\circ \widetilde{T_S}(F)(\omega)=\frac{1}{|S|}\sum_{s\in S}f(\alpha_s(\omega)),$$  which is a locally integrable function on $\Omega$. (Note that we write $\widetilde{\epsilon}_0$ above because $0$ is the identity element of $\zn$).

\section{Rearrangement invariant Banach function spaces}\label{S:BFS}

\subsection{Basic definitions and constructions}

We start by recalling the definition of a rearrangement invariant Banach function space (hereafter referred to as a r.\@i.\@BFS) over a resonant measure space $(\Omega,\mu)$.  Our source for this material is mainly \cite{besh}. 

Given an a.e. finite measurable function $f$ on $(\Omega,\mu)$, we have already defined the distribution function $s\mapsto m(f,s)=\mu(\{\omega\in\Omega:|f(\omega)|>s\}$. The decreasing rearrangement $f^*(s)$ is defined as $$f^*(s)=\sup\{t: m(f,t)\leq s\}.$$ Note that if $f$ and $g$ are equimeasurable functions on $(\Omega,\mu)$ then $f^*(t)=g^*(t)$ for all $t\geq 0$. By \cite[Definition 2.2.3]{besh}, the space $(\Omega,\mu)$ is said to be \textit{resonant} if for each measurable finite a.e.\@ functions $f$ and $g$, the identity $$\int_0^\infty f^*(t)g^*(t)~dt=\sup\int_\Omega|f\widetilde{g}|~d\mu$$ holds as $\widetilde{g}$ ranges over all functions equimeasurable with $g$. We remind the reader that a measure space is resonant if it is either non-atomic or completely atomic with all atoms having the same measure - see eg. \cite[Theorem 2.2.7]{besh}.

One also defines a primitive maximal operator $f\mapsto f^{**}$ as $$f^{**}(s)=\frac{1}{s}\int_0^sf^*(t)~dt.$$ We call $f^{**}$ the \textit{double decreasing rearrangement} of $f$.

When the function norm $\rho$ that defines the Banach function spaces has the property that $\rho(f)=\rho(g)$ for all equimeasurable functions $f$ and $g$, the Banach space is called \textit{rearrangement invariant} - see \cite[Definitions 1.1, 4.1]{besh}.

For any r.\@i.\@BFS $X$ we define another r.\@i.\@BFS $X'$, called the associate space, to be the subset of the a.e.-finite measurable functions $f$ on $(\Omega,\mu)$ for which $\|f\|_{X'}$ is finite, where $$\|f\|_{X'}=\sup\big\{\big|\int_{\Omega}f(\omega)g(\omega)~d\mu(\omega)\big|: g\in X,~\|g\|_X\leq 1\big\}.$$

We shall also have need of another Banach space $X_b\subseteq X$, which is the closure in $X$ of the set of all simple functions in $X$. This is not in general a r.\@i.\@BFS itself, but is useful for the role it plays in the duality theory.

Associated with any rearrangement invariant BFS $X$, there is a \textit{fundamental function} defined by $$\varphi_X(t)=\|\chi_E\|_X \mbox{ for all } t \mbox{ such that there exists a measurable set} E\subset \Omega$$ $$\mbox{with }\mu(E)=t.$$By the rearrangement invariance of $X$, this function is well-defined. It is a \textit{quasi-concave} function, as explained in \cite[Definition 2.5.6]{besh}. Although a priori this function may not be defined on all of $[0,\infty)$, it is always possible to extend $\varphi_X$ to a quasi-concave function defined on all of $[0,\infty)$. If say $\mu(\Omega)$ is finite with $\varphi_X$ defined on $[0, \mu(\Omega))$, then for any $\alpha\in[0,\frac{\|\chi_\Omega\|}{\mu(\Omega)}]$, one may extend $\varphi_X$ by defining its action on $[\mu(\Omega), \infty)$ by means of the formula $\varphi_X(t)=\alpha(t-\mu(\Omega))+\|\chi_\Omega\|_X$. If on the other hand $(\Omega, \mu)$ is discrete with atoms having a measure of $c$, and with $\varphi_X$ say defined on $c\mathbb{N}$, then one simply assigns $\varphi_X$ the value $(1-\lambda)\varphi(cn)+\lambda\varphi_X(c(n+1))$ at any point $(1-\lambda)cn+\lambda c(n+1)$ ($0<\lambda<1$) in the interval $[cn, c(n+1)]$. It is an exercise to see that the resultant extension will in either of these cases still be quasi-concave. Hence in the rest of the paper we may without loss of generality restrict attention to quasi-concave functions defined on all of $[0,\infty)$.  Such functions are automatically subadditive and continuous on $(0,\infty)$. We denote by $\varphi^*_X$ the associate fundamental function of $\varphi_X$, where $\varphi^*_X=\varphi_{X'}$. We shall often make use of the identity $$\varphi_X(t)\varphi^*_X(t)=t$$ for all $t$ in the range of $\mu$, as proved in \cite[Theorem 2.5.2]{besh}. In cases where it is necessary to extend $\varphi_X$ as described above, one may extend $\varphi^*_X$ in a way which preserves the above equality, by simply setting $\varphi^*_X(t)=\frac{t}{\varphi_X(t)}$ for all $t>0$, with $\varphi^*_X(0)=0$. 

Recall that a Young's function is a  convex, nondecreasing function $\Phi:[0,\infty]\to [0,\infty]$ for which  $\Phi(0)=0$, $\lim_{x\to\infty}\Phi(x)=\infty$ and which is neither identically zero nor infinite valued on all of $(0,\infty)$.

One class of spaces that we shall study is that of Orlicz spaces. The theory of these important spaces, which include the standard $L^p$-spaces, is developed in \cite{raore1}. In \cite{besh} and \cite{on}, they are also studied in some depth. We bring to mind the most salient features of their construction.
The Luxemburg norm $\|\cdot\|_{L(\Phi)}$ is defined by a Minkowski functional on the set of all finite a.e.\@ measurable functions on $(\Omega,\mu)$ by the formula $$\|f\|_{L(\Phi)}=\inf\big\{k^{-1}:\int_{\Omega}\Phi(k|f|)~d\mu\leq 1\big\}.$$ The set of all $f$ for which $\|f\|_{L(\Phi)}<\infty$ is the Orlicz space $L(\Phi)$.

There is also another Young's function, called the complementary Young's function. This is the function $\Psi$ defined by $$\Psi(x)=\sup_{y>0}\{xy-\Phi(y)\}.$$ Using this complementary Young's function, it is possible to define another, equivalent norm on the Orlicz space $L(\Phi)$. To this end, define the \textit{Orlicz norm} $\|\cdot\|^{L(\Phi)}$ on the space of measurable functions $f$ on $(\Omega,\mu)$ by setting \begin{equation}\label{E:Orliczdef}\|f\|^{L(\Phi)}=\sup\big\{\int_0^{\infty}f^*(s)g^*(s)~ds: \|g\|_{L(\Psi)}\leq1\big\}.\end{equation}

Now the Orlicz and Luxemburg norms on the Orlicz space $L(\Phi)$ are equivalent. In fact it is proved in \cite[Theorem 4.8.14]{besh} that \begin{equation}\label{E:orliczLux}\|f\|_{L(\Phi)}\leq\|f\|^{L(\Phi)}\leq 2\|f\|_{L(\Phi)}.\end{equation}

Note that for an Orlicz space $L(\Phi)$ equipped with the Luxemburg norm, its fundamental function $\varphi$ is related to $\Phi$ by the equation \begin{equation}\label{E:YoungFund}\varphi(t)=1/\Phi^{-1}(1/t)\end{equation} for all $0<t\leq |\Omega|$ as shown in \cite[Lemma 4.8.17]{besh}. In the sequel, given a Young's function $\Phi$ we define the \textit{fundamental function associated to} $\Phi$ to be the quasiconcave function defined by (\ref{E:YoungFund}). Likewise, given a fundamental function $\varphi$, we may speak of the \emph{associated Young's function} $\Phi$, where again we mean that $\varphi$ and $\Phi$ are related by equation (\ref{E:YoungFund}).

Given a r.\@i.\@BFS $X$ over $(\Omega,\mu)$, we canonically associate two other r.i.\@ BFSs with $X$, apart from the Orlicz space. If $\varphi_X$ is the fundamental function associated with $X$, define the first space $M(X)$ to consist of all the measurable functions $f$ over $\Omega$ such that $$\|f\|_{M(X)}=\sup_{s>0}f^{**}(s)\varphi_X(s)<\infty.$$ The space $M(X)$ is a Banach space with the norm $\|\cdot\|_{M(X)}$. This is the largest r.\@i.\@BFS with fundamental function $\varphi_X$. In other words, if $Y$ is any other r.i\@. BFS with fundamental function $\varphi_X$, then $Y$ is contractively embedded in $M(X)$. Note that as the quasiconcave function $\varphi_X$ is the only property of $X$ required to construct $M(X)$, we may just as well denote this space by $M(\varphi_X)$.

The second space is the associate of $M(X')$. We define $\Lambda(X)$ to consist of all measurable functions $f$ over $\Omega$ such that $$\|f\|_{\Lambda(X)}=\sup\big\{\int_0^{\infty}f^*(s)g^*(s)~ds:~\|g\|_{M(\varphi^*_X)}\leq 1\big\}<\infty.$$ The set $\Lambda(X)$  is a Banach space with norm $\|\cdot\|_{\Lambda(X)}$. (This definition differs somewhat from the one given in \cite[Definition 2.5.12]{besh} for the same space. However as can be seen from \cite[Ch 4, exercise 21 (b)\&(d)]{besh}, the two definitions produce equivalent normings of the same space. The locally cited definition has the advantage of not first having to reduce to the case where $\varphi_X$ is concave for its formulation.) 

As in the case of $M(X)$, we can just as well write $\Lambda(\varphi_X)$, as $\varphi_X$ is the only property of $X$ employed in the construction of $\Lambda(X)$. This is the smallest r.\@i.\@BFS with fundamental function $\varphi_X$; if $Y$ is any other r.\@i.\@BFS with this fundamental function then there is a continuous injection of $\Lambda(X)$ into $Y$.

There is another function space, denoted $M^*(\varphi_X)$,  that can be constructed from the fundamental function of a given r.\@i.\@BFS $X$. In general, this is space is a quasi-Banach space rather than a Banach space, and consists of all those finite a.e.\@ measurable functions $f$ on $(\Omega,\mu)$ for which the (quasi-)norm $\|\cdot\|_{M^*(\varphi_X)}$ defined by $$\|f\|_{M^*(\varphi_X)}=\sup_{s>0}f^*(s)\varphi_X(s)$$ is finite. Although $M^*(\varphi_X)$ is not a r.\@i.\@BFS as such, it does belong to the slightly more general category of rearrangement invariant quasi-Banach Function spaces. (See for example \cite{qxu} for some background on this category.) It is for the sake of emphasising the unity of the theory, that we choose to slightly abuse notation by denoting a quasi-norm by the symbol $\|\cdot\|_{M^*(\varphi_X)}$. Again, note that the only property of $X$ required for this construction is its fundamental function $\varphi_X$. To see that $M^*(\varphi_X)$ is a quasi-normed space, notice that $\|f\|_{M^*(\varphi_X)}=0$ if and only if $f=0$ a.e., $\|\lambda f\|_{M^*(\varphi_X)}=|\lambda|\|f\|_{M^*(\varphi_X)}$ for all complex $\lambda$, and that $$\|f+g\|_{M^*(\varphi_X)}\leq 2(\|f\|_{M^*(\varphi_X)}+\|g\|_{M^*(\varphi_X)})$$ for all $f,g\in M^*(\varphi_X)$. This space was introduced in \cite{shlax} - see also \cite[Ch. 4, exercise 21]{besh}.

We provide a useful equivalent definition of the $M^*(\varphi_X)$-quasi-norm. The proof easily follows from a minor modification of the proof of \cite[Proposition 1.4.5.16, p46]{gr1}.

\begin{lemma}\label{lem2p22} If $\varphi$ is a fundamental function, then $\displaystyle\sup_{t>0}f^*(t)\varphi(t)=\sup_{s>0}s\varphi(m(f,s))$.\end{lemma}

Let us make some remarks on operators between function spaces, in particular on the weak type of an operator. There are two standard definitions of this concept. Let $X$ and $Y$ be rearrangement invariant BFSs. We say that a sublinear operator $T$ has \textit{Marcinkiewicz weak type} $(X,Y)$ if $T$ maps $X$ into $M^*(\varphi_Y)$ and that $T$ has \textit{Lorentz weak type} $(X,Y)$ if it maps $\Lambda(X)$ into $M^*(\varphi_Y)$. Clearly if $T$ is of Marcinkiewicz weak type $(X,Y)$ then it is of Lorentz weak type $(X,Y)$. In the sequel, we shall write `weak type' for `Marcinkiewicz weak type.'

For an operator $T$ of weak type $(X,Y)$ one can write $\|Tf\|_{M^*(\varphi_Y)}\leq c\|f\|_X$. The smallest value of $c$ for which this equation holds is called the norm of $T$.

As we shall mostly be working with $\Lambda$-, $M$- and Orlicz-spaces, let us fix some terminology for dealing with the weak types associated with these kinds of spaces.

\begin{definition}\label{D:LMOweaktype}\textbf{(Weak type)} Let $\Phi_A$ and $\Phi_B$ be Young's functions with associated fundamental functions $\varphi_A$ and $\varphi_B$ respectively. We say that a sublinear operator $T$ has $\Lambda$-, $M$- or $L$-weak type $(\varphi_A,\varphi_B)$ if it respectively maps $\Lambda(\Phi_A)$, $M(\Phi_A)$ or $L(\Phi_A)$ into $M^*(\varphi_B)$.\end{definition}

Bear in mind that if an operator is of $M$-weak type $(\varphi_A,\varphi_B)$, then it is automatically of $L$- and $\Lambda$-weak types $(\varphi_A,\varphi_B)$ too.

We shall often need to work with $L$-, $M$- or $\Lambda$- spaces with the same fundamental function $\varphi$, but over different measure spaces. To specify which measure space is being used, we shall write $\Lambda(\varphi;\Omega)$ to denote the space $\Lambda(\varphi)$ of functions on $\Omega$. Similarly, we use the notations $L(\varphi;\Omega)$ and $M(\varphi;\Omega)$.

\subsection{Comparison of fundamental functions}

It should be quite clear that a lot rests upon the analysis of the various fundamental functions associated with r.i.\@ BFSs. Indeed, based  on techniques for comparing quasiconcave functions developed in the sequel, we will derive our results on the weak type of the transferred operator.

The growth properties of a Young's function have great bearing on the properties of the associated Orlicz space. The same holds more generally for the growth properties of a fundamental function and its associated BFSs. This part of the work will be concerned with translating some standard growth conditions on Young's functions into conditions on fundamental functions. Then we shall analyse these conditions in terms of inequalities prominent in O'Neil's work \cite{on}.

Recall from \cite{raore1} that a Young's function $\Phi$ is said to satisfy the $\Delta_2$ condition , denoted $\Phi\in\Delta_2$, if \begin{eqnarray*}\Phi(2x)\leq K\Phi(x), &&x\geq 0\end{eqnarray*} for some constant $K>0$. The Young's function $\Phi$ satisfies the $\nabla_2$ condition, denoted $\Phi\in\nabla_2$, if \begin{eqnarray*}\Phi(x)\leq\frac{1}{2\ell}\Phi(\ell x),&&x\geq 0\end{eqnarray*} for some $\ell>1$. A subclass of the Young's functions satisfying the $\Delta_2$ condition, are those satisfying the so-called $\Delta'$ condition. Specifically a Young's function $\Phi$ satisfies the $\Delta'$ condition if \begin{eqnarray*}\Phi(xy)\leq K\Phi(x)\Phi(y), &&x,y\geq 0\end{eqnarray*} for some constant $K>0$. The class of Orlicz spaces associated with such Young's functions embody many of the properties of the classical $L^p$-spaces.

We now state these definitions in terms of fundamental functions.

\begin{definition}\label{D:Delta2Fund} A fundamental function $\varphi$ is said to satisfy the $\Delta_2$ condition globally, denoted $\varphi\in\Delta_2$, if \begin{eqnarray}\label{E:delta2}\varphi(Kx)\geq 2\varphi(x),&&x\geq 0\end{eqnarray} for some constant $K>0$. 

A fundamental function $\varphi$ is said to satisfy the $\nabla_2$ condition globally, denoted $\varphi\in\nabla_2$, if \begin{eqnarray}\label{E:nabla2}\varphi(x)\geq \frac{2}{\ell}\varphi(\ell x),&&x\geq 0\end{eqnarray} for some constant $\ell>1$. 
\end{definition} 

\begin{proposition}\label{P:Delta2growth} Suppose that the quasiconcave function $\varphi$ satisfies the $\Delta_2$ condition globally. Then there exists a constant $\epsilon>0$ such that for all $y\leq 1$, $0<p\leq\epsilon$ and $x\in\rn^+$, we have $$\varphi(yx)\leq 2y^p\varphi(x).$$

Furthermore, if $\varphi$ satisfies the $\nabla_2$ condition globally, then there exists a constant $\epsilon >0$ such that for all $y\leq 1$,  $\epsilon\leq p$ and $x\in\rn^+$, it holds that $$\varphi(yx)\geq \frac{1}{2}y^p\varphi(x).$$\end{proposition}

\begin{proof} Suppose that $\varphi$ satisfies the $\Delta_2$ condition globally. Recall that there exists a concave fundamental function $\widetilde{\varphi}$ for which $\varphi\leq \widetilde{\varphi} \leq 2\varphi$ \cite[Proposition 2.5.10]{besh}. It is an exercise to see that $\widetilde{\varphi}$ then also satisfies $\Delta_2$ condition globally. Let $K>1$ be given such that $2\widetilde{\varphi}(s)\leq \widetilde{\varphi}(Ks)$ for all $s\geq 0$, and let $\phi$ be the derivative of $\widetilde{\varphi}$. For fixed $x,y>0$ we then have that $$\widetilde{\varphi}(xy)\leq \widetilde{\varphi}(Kxy)-\widetilde{\varphi}(xy) =\int_x^{Kx}y\phi(ty)\,dt \leq [(K-1)x]y\phi(xy)$$where in the last inequality we used the fact that $\phi$ is nonincreasing. Let $x>0$, $y\geq1$, and $0<q\leq 1$ be given. From what we have just shown, it follows that 
$$\log\frac{\widetilde{\varphi}(xy)}{\widetilde{\varphi}(x)} =\int_{1}^{y}x\frac{\phi(xt)}{\widetilde{\varphi}(xt)}\,dt \geq \frac{1}{K-1}\int_{1}^{y}\frac{1}{t}\,dt\geq \frac{q}{K-1}\int_{1}^{y}\frac{1}{t}\,dt = \frac{q}{K-1}\log y.$$With $p=\frac{q}{K-1}$ this leads to the conclusion that $\frac{\widetilde{\varphi}(xy)}{\widetilde{\varphi}(x)}\geq y^p$ for all $x>0$, $y\geq 1$, and all $0<p\leq \frac{1}{K-1}$. Set $\epsilon=\frac{1}{K-1}$. For $0<y\leq 1$ and $p\leq \epsilon$, we use the change of variables $y'=1/y$ and $x'=xy$, to conclude that $\widetilde{\varphi}(x')\geq \left(\frac{1}{y'}\right)^p \widetilde{\varphi}(x'y')$. That is, $\widetilde{\varphi}(xy)\leq y^p\widetilde{\varphi}(x)$ for all $x\geq 0$, $0\leq y\leq 1$. Therefore $$\varphi(yx)\leq\widetilde{\varphi}(yx)\leq y^p\widetilde{\varphi}(x)\leq 2y^p\varphi(x),$$ thereby proving the claim.

Next suppose that $\varphi$ satisfies the $\nabla_2$ condition. It is then a simple exercise to see that $\varphi^*(t)=\frac{t}{\varphi(t)}$ ($t>0$) is then a fundamental function satisfying the $\Delta_2$ condition globally. Thus from what we proved above, there exists a constant $1>\epsilon'>0$ such that for all $0<y\leq 1$, $0\leq p'\leq \epsilon'$ and $x\in\rn^+$, we have $\frac{xy}{{\varphi}(xy)}={\varphi}^*(xy)\leq 2y^{p'}{\varphi}^*(x)=2y^{p'}\frac{x}{{\varphi}(x)}$. This in turn reduces to the inequality $2{\varphi}(xy) \geq y^{1-p'}{\varphi}(x)$ for all $0<y\leq 1$, $0\leq p'\leq \epsilon'$ and $x\in\rn^+$. Set $\epsilon:=1-\epsilon'>0$ and $p:=1-p'$ to get $$y^p\varphi(x)\leq 2\varphi(xy),$$ as required.\end{proof}

The inequalities $\varphi(yx)\geq \frac{1}{2}y^p\varphi(x)$ and $\varphi(yx)\leq 2y^p\varphi(x)$ obtained in the Proposition above are instances of \textit{tail growth} conditions: they are valid for $y$ small enough. A larger class of fundamental functions satisfying such conditions is provided in \cite{shlax}, which we now recall.

\begin{definition}[\cite{shlax}]\label{D:LU}We define two classes of fundamental functions as follows.\begin{enumerate}\item $\varphi\in\mathcal{U}$ if for some $0<\alpha<1$, there are positive constants $A$ and $\delta$ such that \begin{eqnarray*}\varphi(ts)&\geq &At^{\alpha}\varphi(s)\mbox{ if }t\leq\delta.\end{eqnarray*}
The $\mathcal{U}$-index of $\varphi$, denoted $\rho^\varphi_{\mathcal{U}}$, is the infimum of all $\alpha$  for which the above inequality obtains.
\item $\varphi\in \mathcal{L}$ if for some $\alpha>0$, there are positive constants $A$ and $\delta$ such that \begin{eqnarray*}\varphi(ts)&\leq& At^{\alpha}\varphi(s)\mbox{ if }t\leq\delta.\end{eqnarray*}
The $\mathcal{L}$-index of $\varphi$, denoted $\rho^\varphi_{\mathcal{L}}$, is the supremum of all $\alpha$  for which the above inequality obtains.\end{enumerate}\end{definition}

Proposition \ref{P:Delta2growth} shows that if $\varphi$ satisfies the $\Delta_2$ condition, then $\varphi\in\mathcal{L}$ and if $\varphi$ satisfies the $\nabla_2$ condition then $\varphi\in\mathcal{U}$. We pause to note that for a r.\@i.\@BFS $X$ with fundamental function $\varphi$, on setting $M(t,X)=\sup_{s>0}\varphi(st)/\varphi(s)$, Zippin \cite{zip} defined the \textit{fundamental indices} as \begin{eqnarray*}\underline{\beta}_X=\lim_{t\to 0^+}\frac{\ln M(t,X)}{\ln t},&&\overline{\beta}_X=\lim_{t\to \infty}\frac{\ln M(t,X)}{\ln t}.\end{eqnarray*}These indices are further discussed and analysed in \cite{shlax}.

\subsection{Estimates of integrals and function norms}

When working with maximal inequalities, there are certain integrals that we will need to estimate. The following Proposition covers the cases that we will need.

First, some terminology, following \cite[\S 10]{ke}: consider a measure space $(\Omega,\Sigma,\mu)$ and a countable collection $\mathcal{D}\subset\Sigma$ of measurable subsets of $\mu$-finite measure. The $\sigma$-algebra $\sigma(\mathcal{D})$ generated by $\mathcal{D}$ is contained in $\Sigma$. If for any $F\in\Sigma$ there is a $D\in\sigma(\mathcal{D})$ such that $F\Delta D$ has null measure, where $F\Delta D$ denotes the symmetric difference between $D$ and $F$, we say that $(\Omega,\Sigma,\mu)$ is \textit{countably generated modulo null sets}, or just \textit{countably generated}. We call $\mathcal{D}$ the \textit{generators} of $\Sigma$. Moreover, we may assume that $\mathcal{D}$ is an algebra, for if $\mathcal{D}$ is countable, so is the algebra it generates. If $\mathcal{D}$ is an algebra of sets that generates $\Sigma$ in the above sense, it is easy to see that if $F\subset\Omega$ is any $\mu$-finite subset and $\epsilon>0$, then there is a $D\in\mathcal{D}$ such that $\mu(D\Delta F)<\epsilon$ and $|\mu(D)-\mu(F)|<\epsilon$. 

\begin{proposition}\label{P:XinY(Z)} Let $(\Omega_1,\mu_1)$ and $(\Omega_2,\mu_2)$ be resonant spaces with $\Omega_2$ countably generated. Let $\Phi_A, \Phi_B$ and $\Phi_C$ be Young's functions and $\varphi_A,\varphi_B$ and $\varphi_C$ be their respective associated fundamental functions satisfying \begin{equation}\label{E:XinY(Z)}\theta\varphi_A(st)\geq\varphi_B(s)\varphi_C(t)\end{equation} for all $s,t>0$ and some $\theta>0$. Let $f$ be a measurable function on $\Omega_1\times\Omega_2$ and $E\subset\Omega_1$ a subset of finite measure. 

\begin{enumerate}\item[1)] If $f\in M(\Phi_A)$, then
\begin{eqnarray*}\frac{\varphi_C(|E|)}{|E|}\int_E\|f_{\omega_1}\|_{M(\Phi_B)}~d\mu_1(\omega_1)&\leq&4e^3\theta\|f\|_{M(\Phi_A)}.\end{eqnarray*} \item[2)] If $f\in\Lambda(\Phi_A)$ and $\lim_{t\to 0}\varphi^*_B(t)=0$, then
\begin{eqnarray*}\frac{\varphi_C(|E|)}{|E|}\int_E\|f_{\omega_1}\|_{\Lambda(\Phi_B)}~d\mu_1(\omega_1)&\leq&6\theta\|f\|_{\Lambda(\Phi_A)}.\end{eqnarray*}
\item[3)] If $f\in L(\Phi_A)$ and $\lim_{t\to 0}\varphi^*_B(t)=0$, then 
\begin{eqnarray*}\frac{\varphi_C(|E|)}{|E|}\int_E\|f_{\omega_1}\|_{L(\Phi_B)}~d\mu_1(\omega_1)&\leq& \theta\|f\|_{L(\Phi_A)}.\end{eqnarray*}
\end{enumerate}\end{proposition}

As the proof of this Proposition relies heavily on \cite[Theorem 8.18]{on}, it is worth mentioning that the condition on the fundamental functions given there, namely $\Phi_A^{-1}(st)\Phi_B^{-1}(t)\leq\theta t\Phi_C^{-1}(s)$, can with the help of (\ref{E:YoungFund}) and the identity $\varphi_B(t)\varphi^*_B(t)=t$ be written in the equivalent form $$\theta\varphi_A(st)\geq\varphi^*_B(t)\varphi_C(s).$$

\begin{proof}  Let $\mathcal{D}$ be a countable algebra that generates $(\Omega_2,\mu_2)$. 

Suppose $f\in M(\Phi_A)$. For any subset $\Delta\subset\Omega_2$ of finite measure, define $h_{\Delta}$ by $$h_{\Delta}(\omega_1)=\frac{1}{\varphi^*_B(|\Delta|)}\int_{\Delta}|f(\omega_1,\omega_2)|~d\mu_2(\omega_2).$$ Thus $\displaystyle h_\Delta(\omega_1)=\int_{\Omega_2}|f|(\chi_\Delta/\varphi^*_B(|\Delta|))~d\mu_2$. Note that (\ref{E:XinY(Z)}) can be written in the form $$\theta\varphi_A(st)\geq(\varphi_B^*)^*(s)\varphi_C(t)$$ because for any fundamental function $\varphi$, $(\varphi^*)^*=\varphi$. We apply \cite[Theorem 8.18, part $1^\circ$]{on} to conclude that $h_{\Delta}\in M(\Phi_C)$, with $\|h_{\Delta}\|_{M(\Phi_C)}\leq 4e^3\theta\|f\|_{M(\Phi_A)}.$ We also used the obvious fact that $\|\chi_\Delta/(\varphi^*_B(|\Delta|))\|_{\Lambda(\varphi^*_B)}=1$.

Now define $$\tilde{h}=\sup_{\Delta\in\mathcal{D}}h_{\Delta}.$$ As $\tilde{h}$ is the supremum of a countable number of functions, it is itself a measurable function.

For any $\Delta\in\mathcal{D}$ and $\mu_1$-almost every $\omega_1\in\Omega_1$, \begin{eqnarray*}\widetilde{h}(\omega_1)&=&\frac{1}{\varphi^*_B(|\Delta|)}\int_{\Delta}|f_{\omega_1}|~d\mu_2=\frac{1}{|\Delta|}\int_\Delta|f_{\omega_1}|~d\mu_2.\varphi_B(|\Delta|)\\&\leq&f_{\omega_1}^{**}(|\Delta|)\varphi_B(|\Delta|)\leq\|f_{\omega_1}\|_{M(\Phi_B)},\end{eqnarray*} by definition of the norm $\|\cdot\|_{M(\Phi_B)}$. Hence $\widetilde{h}(\omega_1)\leq\|f_{\omega_1}\|_{M(\Phi_B)}$ a.e.

On the other hand for any fixed $\epsilon>0$, there is a $t>0$ such that $f^{**}_{\omega_1}(t)\varphi_B(t)>\|f_{\omega_1}\|_{M(\Phi_B)}-\epsilon$. As $(\Omega_2,\mu_2)$ is a resonant space, by \cite[Proposition 2.3.3]{besh}, there is a subset $F$ such that $|F|=t$ and $$\frac{1}{|F|}\int_F|f_{\omega_1}|~d\mu_2>f^{**}_{\omega_1}(t)-\epsilon/\varphi_B(t).$$ Hence $$\frac{1}{\varphi^*_B(|F|)}\int_F|f_{\omega_1}|~d\mu_2>f^{**}_{\omega_1}(t)\varphi_B(t)-\epsilon.$$ Because $\mathcal{D}$ is dense in the Borel $\sigma$-algebra, there is a $\Delta\in\mathcal{D}$ such that $$\bigg|\frac{1}{\varphi^*_B(|\Delta|)}\int_\Delta|f_{\omega_1}|~d\mu_2-\frac{1}{\varphi_B^*(|F|)}\int_F|f_{\omega_1}|~d\mu_2\bigg|<\epsilon.$$ Therefore \begin{eqnarray*}h_{\Delta}(\omega_1)&=&\frac{1}{\varphi^*_B(|\Delta|)}\int_\Delta|f_{\omega_1}|~d\mu_2\\&>&\frac{1}{\varphi^*_B(|F|)}\int_F|f_{\omega_1}|~d\mu_2-\epsilon\\&>&f^{**}_{\omega_1}(t)\varphi_B(t)-2\epsilon>\|f_{\omega_1}\|_{M(\Phi_B)}-3\epsilon,\end{eqnarray*}whence\begin{eqnarray*}\widetilde{h}(\omega_1)&>&\|f_{\omega_1}\|_{M(\Phi_B)}-3\epsilon.\end{eqnarray*} As $\epsilon>0$ was arbitrary, it is clear that $\widetilde{h}(\omega_1)\geq\|f_{\omega_1}\|_{M(\Phi_B)}$. So
we have proved that $\tilde{h}(\omega_1)=\|f_{\omega_1}\|_{M(\Phi_B)}$ for almost all $\omega_1\in\Omega_1$. Also, $\|\tilde{h}\|_{M(\Phi_C)}\leq 4e^3\theta\|f\|_{M(\Phi_A)}$ because as we have already shown, $\|h_{\Delta}\|_{M(\Phi_C)}\leq 4e^3\theta\|f\|_{M(\Phi_A)}$ for all $\Delta\in\mathcal{D}$. Combining these two facts yields part 1) of the Proposition.

For the second part, we shall follow a similar strategy to that of the first part. Consider the space $M(\varphi_B^*)_b$ over $\Omega_2$, which is the closure of the space of all simple functions in $M(\varphi_B^*)$ whose support has finite measure. The condition $\lim_{t\to 0}\varphi^*_B(t)=0$ means that by \cite[Theorem 2.5.5]{besh},  $M(\varphi_B^*)_b$ is separable and that $(M(\varphi_B^*)_b)^*=\Lambda(\Phi_B).$ Let $\mathcal{D}$ be a countable dense subset of the unit ball of $M(\varphi_B^*)_b$. By the above remarks, this is a norming set for $\Lambda(\Phi_B)$, in that for any $g\in\Lambda(\Phi_B)$, we have $$\|g\|_{\Lambda(\Phi_B)}=\sup_{\delta\in\mathcal{D}}\int_{\Omega_2}|g(\omega_2)\delta(\omega_2)|~d\mu_1(\omega_2).$$

Now for each $\delta\in\mathcal{D}$, define the functions \begin{eqnarray*}h_{\delta}(\omega_1)&=&\int_{\Omega_2}|f(\omega_1,\omega_2)\delta(\omega_2)|~d\mu_2(\omega_2)\\\tilde{h}(\omega_1)&=&\sup_{\delta\in\mathcal{D}}h_{\delta}(\omega_1).\end{eqnarray*} Note that as $\tilde{h}$ is the supremum of a countable number of measurable functions, it is itself measurable.

By \cite[Theorem 8.18, part $3^{\circ}$]{on}, $\|h_{\delta}\|_{L(\Phi_C)}\leq 6\theta\|f\|_{\Lambda(\Phi_A)}\|\delta\|_{M(\varphi_B^*)}\leq 6\theta\|f\|_{\Lambda(\Phi_A)}.$ Hence $\|\tilde{h}\|_{L(\Phi_C)}\leq 6\theta\|f\|_{\Lambda(\Phi_A)}.$

On the other hand, for each $\omega_1\in\Omega_1$, \begin{eqnarray*}\tilde{h}(\omega_1)&=&\sup_{\delta\in\mathcal{D}}\int_{\Omega_2}|f(\omega_1,\omega_2)\delta(\omega_2)|~d\mu_2(\omega_2)\\&=&\|f_{\omega_1}\|_{\Lambda(\Phi_B)}\end{eqnarray*} where the last equality is true on account of $\mathcal{D}$ being a norming subset of $M(\varphi_B^*)_b$ for $\Lambda(\Phi_B)$.

Hence if $E\subset\Omega_1$ is any set of finite measure, then by H\"older's inequality\begin{eqnarray*}\frac{\varphi_C(|E|)}{|E|}\int_E\|f_{\omega_1}\|_{\Lambda(\Phi_B)}~d\mu_1(\omega_1)&\leq&\frac{\varphi_C(|E|)}{|E|}\|\tilde{h}\|_{L(\Phi_C)}\|\chi_E\|_{L(\Phi^*_C)}\\&=&\|\tilde{h}\|_{L(\Phi_C)}\leq 6\theta\|f\|_{\Lambda(\Phi_A)},\end{eqnarray*} proving part 2).

For the third part, let $\Psi_B$ denote the Young's function complementary to $\Phi_B$ and note that because $L(\Phi_B)$ is an Orlicz space with the Luxemburg norm, its associate space is the Orlicz space $L(\Psi_B)$ under the Orlicz norm and with fundamental function $\varphi^*_B$. 

The proof now proceeds as for the second part. Because $\lim_{t\to\infty}\varphi^*_B(t)=0$, by \cite[Theorem 2.5.5]{besh},  $L(\Psi_B)_b$ is separable and that $(L(\Psi_B)_b)^*=L(\Phi_B).$ Let $\mathcal{D}$ be a countable dense subset of the unit ball of $L(\Psi_B)_b$ and define as before the functions $h_\delta$ and $\widetilde{h}$. By \cite[Theorem 8.18, part $2^{\circ}$]{on}, $$\|h_\delta\|_{L(\Phi_C)}\leq\theta\|f\|_{L(\Phi_A)}\|\delta\|_{L(\Phi_B)}\leq \theta\|f\|_{L(\Phi_A)}\|\delta\|^{L(\Phi_B)}\leq \theta\|f\|_{L(\Phi_A)},$$ where $\|\cdot\|_{L(\Phi_B)}$ and $\|\cdot\|^{L(\Phi_B)}$ denote the Luxemburg and Orlicz norms respectively and we used the fact that by (\ref{E:orliczLux}), $\|\delta\|_{L(\Phi_B)} \leq \|\delta\|^{L(\Phi_B)}\leq 1$. Hence $\|\widetilde{h}\|_{L(\Phi_C)}\leq\theta\|f\|_{L(\Phi_A)}$.

For each $\omega_1\in\Omega_1$, \begin{eqnarray*}\tilde{h}(\omega_1)&=&\sup_{\delta\in\mathcal{D}}\int_{\Omega_2}|f(\omega_1,\omega_2)\delta(\omega_2)|~d\mu_2(\omega_2)\\&=&\|f_{\omega_1}\|_{L(\Phi_B)}\end{eqnarray*} where the last equality is true on account of $\mathcal{D}$ being a norming subset of $L(\Psi_B)_b$ for $L(\Phi_B)$.

For a subset $E\subset\Omega_1$ of finite measure, H\"older's inequality reveals that\begin{eqnarray*}\frac{\varphi_C(|E|)}{|E|}\int_E\|f_{\omega_1}\|_{L(\Phi_B)}~d\mu_1(\omega_1)&\leq&\frac{\varphi_C(|E|)}{|E|}\|\tilde{h}\|_{L(\Phi_C)}\|\chi_E\|_{L(\Phi^*_C)}\leq \theta\|f\|_{L(\Phi_A)},\end{eqnarray*} proving part 3).\end{proof}

\section{The weak type of the transfer operator}\label{S:weakTypeDef}

This Section is devoted to calculating the type of the transfer operator $T^{\#}$ from information on the type of $T$. 

\subsection{Kolmogorov's inequality for r.i.\@ BFSs}

The following theorem will be useful in determining the weak type of an operator. It is an extension of Kolmogorov's criterion as found in \cite[Theorem 3.3.1]{guz}.

\begin{theorem}\label{thKol} Let $(\Omega,\mu)$ be a measure space and let $T$ be an operator of weak type $(X,Y)$ for r.i.\@ BFSs $X$ and $Y$ on $\Omega$, and has norm $c$. Let $\varphi$ be the fundamental function of the space $Y$. If $0<\sigma<1$ and $A$ is any subset of $\Omega$ of finite measure, then for any $f\in X$ we have \begin{eqnarray}\label{kol2}\int_A|Tf|^{\sigma}~d\mu(x)&\leq &\frac{c^{\sigma}}{1-\sigma}\big[\varphi^*(|A|)\big]^{\sigma}|A|^{1-\sigma}\|f\|^{\sigma}_{X}.\end{eqnarray} Conversely, if $T$ satisfies this inequality for some $c$ and $0<\sigma<1$, and for each $f\in X$ and each $A\subset \Omega$ with finite measure, then $T$ is of weak type $(X,Y)$.\end{theorem}

\begin{proof} Suppose $0<\sigma<1$. As $t\mapsto\varphi(t)/t$ is nondecreasing, if $s\leq t$, we have the implications \begin{eqnarray*}\frac{\varphi(s)}{s}\geq\frac{\varphi(t)}{t}&\Longrightarrow&\frac{\varphi(t)}{\varphi(s)}\leq\frac{t}{s}\Longrightarrow\frac{\varphi^{\sigma}(t)}{\varphi^{\sigma}(s)}\leq\big(\frac{t}{s}\big)^{\sigma}\\&\Longrightarrow&\frac{s}{t}\varphi^{\sigma}(t)\leq\big(\frac{t}{s}\big)^{\sigma-1}\varphi^{\sigma}(s).\end{eqnarray*}

Note that  $\chi_{[1,\infty)}(t/s)=\chi_{(0,t]}(s)$ for all $s,t\in\rn^+$. On the multiplicative group of the positive reals, we now compute, using the convolution of the functions $[|Tf|^\sigma]^*(x)\varphi^{\sigma}(x)$ and $x^{\sigma-1}\chi_{[1,\infty)}(x)$ at $t\in\rn$:

\begin{eqnarray*}\frac{\varphi^{\sigma}(t)}{t}\int_0^t[|Tf|^\sigma]^*(s)ds&\leq&\int_0^t[|Tf|^\sigma]^*(s)\varphi^{\sigma}(s)\big(\frac{t}{s}\big)^{\sigma-1}\frac{ds}{s}\\&=&\int_0^{\infty}[|Tf|^\sigma]^*(s)\varphi^{\sigma}(s)\big(\frac{t}{s}\big)^{\sigma-1}\chi_{(0,t]}(s)~\frac{ds}{s}\\&=&\int_0^{\infty}[|Tf|^\sigma]^*(s)\varphi^{\sigma}(s)\big(\frac{t}{s}\big)^{\sigma-1}\chi_{[1,\infty)}(t/s)~\frac{ds}{s}\\&=&[|Tf|^\sigma]^*(x)\varphi^{\sigma}(x)*x^{\sigma-1}\chi_{[1,\infty)}(x)\big|_t.\end{eqnarray*}

With this in hand, we exploit the inequality $$\|[|Tf|^\sigma]^*(x)\varphi^{\sigma}(x)*x^{\sigma-1}\chi_{[1,\infty)}(x)\|_{\infty}\leq\|[|Tf|^\sigma]^*(x)\varphi^{\sigma}(x)\|_{\infty}\|x^{\sigma-1}\chi_{[1,\infty)}(x)\|_1.$$
As $\displaystyle\|x^{\sigma-1}\chi_{[1,\infty)}(x)\|_1=\int_1^{\infty}s^{\sigma-2}ds=\frac{1}{\sigma-1}s^{\sigma-1}\big|_1^{\infty}=\frac{1}{1-\sigma},$ we have 
\begin{eqnarray}\label{E:sal1}\sup_{t>0}\frac{\varphi^{\sigma}(t)}{t}\int_0^t[|Tf|^\sigma]^*(s)ds&\leq&\frac{1}{1-\sigma}\sup_{t>0}[|Tf|^\sigma]^*(t)\varphi^{\sigma}(t)\nonumber\\&=&\frac{1}{1-\sigma}\big[\sup_{t>0}(Tf)^*(t)\varphi(t)\big]^{\sigma}.\end{eqnarray}

Here we used the fact that $(|f|^{\sigma})^*=(|f|^*)^{\sigma}$ for all $0<\sigma<\infty$ (see Prop 2.1.7 p41 of \cite{besh}). We thus have the following inequality. \begin{eqnarray*}\frac{\varphi^{\sigma}(t)}{t}\int_0^t[|Tf|^\sigma]^*(s)ds&\leq&\frac{1}{1-\sigma}\|Tf\|^{\sigma}_{M^*(\varphi_Y)}\\&\leq&\frac{c^{\sigma}}{1-\sigma}\|f\|^{\sigma}_{X},\end{eqnarray*} by our hypothesis. It is obvious that $\displaystyle\int_A|Tf|^{\sigma}~d\mu\leq\int_0^{|A|}[|Tf|^\sigma]^*(s)ds$, and so we obtain \begin{eqnarray*}\int_A|Tf|^{\sigma}~d\mu&\leq&\frac{c^{\sigma}}{1-\sigma}\frac{|A|}{\varphi^{\sigma}(|A|)}\|f\|^{\sigma}_{X}\\&=&\frac{c^{\sigma}}{1-\sigma}\big[\varphi^*(|A|)\big]^{\sigma}|A|^{1-\sigma}\|f\|^{\sigma}_{X}.\end{eqnarray*}  To get the last equality, we used the identity $\varphi(|A|)\varphi^*(|A|)=|A|$.

To prove the converse, suppose that $T$ satisfies (\ref{kol2}) for some $0<\sigma<1$ and fix $\lambda>0$. Consider a set $K\subset\{\omega:|Tf(\omega)|>\lambda\}$ of finite measure. By hypothesis, $$|K|\leq\int_K\frac{|Tf|^{\sigma}}{\lambda^{\sigma}}~d\mu\leq\frac{1}{\lambda^{\sigma}}\frac{c^{\sigma}}{1-\sigma}[\varphi^*(|K|)]^{\sigma}|K|^{1-\sigma}\|f\|^{\sigma}_{X}.$$ Consequently, the following computations are valid:\begin{eqnarray*}\frac{|K|^{\sigma}}{\varphi^*(|K|)^{\sigma}}&\leq&\frac{1}{\lambda^{\sigma}}\frac{c^{\sigma}}{1-\sigma}\|f\|^{\sigma}_X;\\\varphi(|K|)&\leq&\frac{1}{\lambda}\frac{c}{(1-\sigma)^{1/\sigma}}\|f\|_X;\\\lambda\varphi(m(|Tf|,\lambda))&\leq&\frac{c}{(1-\sigma)^{1/\sigma}}\|f\|_X;\\\|Tf\|_{M^*(\varphi_Y)}&\leq&\frac{c}{(1-\sigma)^{1/\sigma}}\|f\|_X;\end{eqnarray*} where in the last  line we used the identity $\|Tf\|_{M^*(\varphi_Y)}=\sup_{\lambda>0}\lambda\varphi_Y(m(|Tf|,\lambda))$ of Lemma \ref{lem2p22}. This proves the converse.\end{proof}

\subsection{Computation of the strong and weak type of $T^{\#}$ for general $\Omega$}

In the rest of the paper, we shall work with dynamical systems $(G,\Omega,\mu,\alpha)$ as given in Definition \ref{D:dynsys}, where $(\Omega,\mu)$ is a $\sigma$-finite and resonant measure space and $G$ is a $\sigma$-finite locally compact group. We shall also ultimately require that $G$ be amenable in order to obtain the most powerful results.

The first set of results we prove hold for general locally compact groups, with minimal restrictions on the spaces involved. However, the price paid for achieving such generality is that most of these results only apply to single operators, not sequences. The one exception to this rule is transfers of convolution operators. Results that apply to sequences of more general classes of operators can be produced by adding the restriction that the group be amenable and restricting the growth rates of the associated fundamental functions.

\begin{lemma}\label{L:preCalderonTheorem1} Let $X,Y$ be r.i.\@ BFSs over $G$ with fundamental functions $\varphi_X$ and $\varphi_Y$ respectively and let $T$ be a transferable operator of  weak type $(X,Y)$. Let $U$ be the open neighbourhood satisfying the conditions in Definition \ref{D:to}(2). Then for any subset $A\subset\Omega$ of finite measure, and $0<\sigma<1$, there is a compact neighbourhood $\widetilde{K}$ of the identity such that\begin{eqnarray*}\frac{1}{|A|}\int_A|T^\#f|^\sigma(\omega) d\mu&\leq&\frac{2c^\sigma}{1-\sigma}[\varphi_Y(|\widetilde{K}|)]^{-\sigma}\bigg(\frac{1}{|A|}\int_A\|(F_{\widetilde{K}U^{-1}})_\omega\|_X~d\mu\bigg)^\sigma,\end{eqnarray*} where for each $\omega\in\Omega$ the cross section $(F_{\widetilde{K}U^{-1}})_{\omega}$ is the measurable function defined on $G$ by $t\mapsto F_{\widetilde{K}U^{-1}}(t,\omega)$.\end{lemma}

Recall from (\ref{E:F_K}) that $$F_{\widetilde{K}U^{-1}}(\omega,t)=\left\{\begin{array}{cc}f(\alpha_t(\omega))&\mbox{if}~t\in \widetilde{K}U^{-1}\\0&\mbox{otherwise.}\end{array}\right.$$
\begin{proof} Let $\widetilde{K}$ be any compact neighbourhood of the identity. We shall first derive an inequality that holds for all such $\widetilde{K}$, and then show how this inequality implies the inequality of the Lemma when $\widetilde{K}$ is chosen sufficiently small. As shown in Definition \ref{R:transfer}, we may view $F'$ and $F'_{\widetilde{K}U^{-1}}$ as members of 
$\mathcal{L}^{\rm loc}(\Omega,C(G))$. By Corollary \ref{cgconcrete}, we may then view $F'$ and $F'_{\widetilde{K}U^{-1}}$ as members of $LK^{\rm r-loc}(\Omega\times G)$, and so we may apply Fubini's theorem.  This theorem, together with Lemma \ref{L:Kdominance} (which 
we employ by identifying $E$ with $\widetilde{K}$ and $K$ with $\widetilde{K}U^{-1}$), yields: 
\begin{eqnarray*}\int_{\widetilde{K}}\int_A |F'(\omega,t)|^\sigma~d\mu dt&=&\int_A\int_{\widetilde{K}} |F'(\omega,t)|^\sigma~dt d\mu
\\&\leq &\int_A\int_{\widetilde{K}}|F'_{\widetilde{K}U^{-1}}(\omega,t)|^\sigma~dt d\mu.\end{eqnarray*}As $T$ is of weak type 
$(X,Y)$, using Kolmogorov's criterion (\ref{kol2}) we have that $$\int_{\widetilde{K}}|F'_{\widetilde{K}U^{-1}}(\omega,t)|^\sigma dt\leq\frac{c^\sigma}{1-\sigma}[\varphi^*_Y(|\widetilde{K}|)]^\sigma|\widetilde{K}|^{1-\sigma}\|(F_{\widetilde{K}U^{-1}})_\omega\|_X^\sigma.$$ From Jensen's inequality and the identity $\varphi_Y^*(t)\varphi_Y(t)=t$, 
\begin{eqnarray*}\int_{\widetilde{K}}\int_A |F'(\omega,t)|^\sigma~d\mu dt&\leq&\frac{c^\sigma}{1-\sigma}[\varphi^*_Y(|\widetilde{K}|)]^\sigma|\widetilde{K}|^{1-\sigma}\int_A\|(F_{\widetilde{K}U^{-1}})_\omega\|_X^\sigma~d\mu\\&\leq&\frac{c^\sigma}{1-\sigma}[\varphi_Y(|\widetilde{K}|)]^{-\sigma}|\widetilde{K}||A|^{1-\sigma}\bigg(\int_A\|(F_{\widetilde{K}U^{-1}})_\omega\|_X~d\mu\bigg)^\sigma.\end{eqnarray*}

We rewrite this as $$\frac{1}{|\widetilde{K}|}\int_{\widetilde{K}}\bigg[\frac{1}{|A|}\int_A|F'(t,\omega)|^\sigma~d\mu\bigg] dt\leq\frac{c^\sigma}{1-\sigma}[\varphi_Y(|\widetilde{K}|)]^{-\sigma}\bigg(\frac{1}{|A|}\int_A\|(F_{\widetilde{K}U^{-1}})_\omega\|_X~d\mu\bigg)^\sigma.$$ As $$t\mapsto\frac{1}{|A|}\int_A|F'(\omega,t)|^\sigma d\mu$$ is continuous, and as $|F'(\omega,1)|=|T^\#f|(\omega)$ by definition, from the Lebesgue Differentiation Theorem we obtain 
\begin{eqnarray*}\frac{1}{|A|}\int_A|T^\#f|^\sigma(\omega) d\mu&=&\lim_{\widetilde{K}\to\{1\}}\frac{1}{|\widetilde{K}|}\int_{\widetilde{K}}\bigg[\frac{1}{|A|}\int_A|F'(\omega,t)|^\sigma~d\mu\bigg] dt.\end{eqnarray*} Hence for some $\widetilde{K}$ small enough, \begin{eqnarray*}\frac{1}{|A|}\int_A|T^\#f|^\sigma(\omega) d\mu&\leq&\frac{2}{|\widetilde{K}|}\int_{\widetilde{K}}\bigg[\frac{1}{|A|}\int_A|F'(\omega,t)|^\sigma~d\mu\bigg] dt\\&\leq&\frac{2c^\sigma}{1-\sigma}[\varphi_Y(|\widetilde{K}|)]^{-\sigma}\bigg(\frac{1}{|A|}\int_A\|(F_{\widetilde{K}U^{-1}})_\omega\|_X~d\mu\bigg)^\sigma.\end{eqnarray*}

\end{proof}

\begin{corr}\label{C:CalderonTh1corr2}Let $(\Omega,\mu,G,\alpha)$ be a dynamical system with $(\Omega,\mu)$ resonant and $(G,h)$ countably generated and resonant. Let $T$ be a transferable operator of weak type $(X,Y)$ and suppose that $\Phi_A$ and $\Phi_B$ are Young's functions with respective associated fundamental functions $\varphi_A$ and $\varphi_B$ satisfying $$\theta\varphi_A(st)\geq\varphi_X(s)\varphi_B(t)$$ for some $\theta>0$ and all $s,t>0$.
\begin{enumerate}\item[1)] If $X$ is an Orlicz space and $\lim_{t\to 0}\varphi^*_X(t)=0$, then $T^{\#}$ is of $L$-weak type $(\varphi_A,\varphi_B)$.
\item[2)] If $X$ is an  $M$- space then $T^{\#}$ is of $M$-weak type $(\varphi_A,\varphi_B)$.
\item[3)] If $X$ is a $\Lambda$-space and $\lim_{t\to 0}\varphi^*_X(t)=0$, then $T^{\#}$ is of $\Lambda$-weak type $(\varphi_A,\varphi_B)$.
\end{enumerate}

\end{corr}

\begin{proof} The other cases being similar, we prove only part 3). Let $A\subset\Omega$ have finite measure and let $U$ be the open neighbourhood guaranteed by Definition \ref{D:to}(2).  Then for any $0<\sigma<1$, by Lemma \ref{L:preCalderonTheorem1}, there is a compact neighbourhood of the identity $K\subset G$ such that
\begin{eqnarray*}\frac{1}{|A|}\int_A|T^\#f|^\sigma(\omega) d\mu&\leq&\frac{2c^\sigma}{1-\sigma}[\varphi_Y(|K|)]^{-\sigma}\bigg(\frac{1}{|A|}\int_A\|(F_{KU^{-1}})_\omega\|_{\Lambda(X)}~d\mu\bigg)^\sigma.\end{eqnarray*}

From Proposition \ref{P:XinY(Z)}, part 2), $$\frac{\varphi_B(|A|)}{|A|}\int_A\|(F_{KU^{-1}})_{\omega}\|_{\Lambda(X)}~d\mu(\omega)\leq 6\theta\|F_{KU^{-1}}\|_{\Lambda(\Phi_A)}.$$ 

Combining these last two inequalities yields
\begin{eqnarray}\label{E:s4new1}\frac{1}{|A|}\int_A|T^\#f|^\sigma(\omega) d\mu&\leq&\frac{2 (6\theta c)^\sigma}{1-\sigma}[\varphi_Y(|K|)\varphi_B(|A|)]^{-\sigma}\|F_{KU^{-1}}\|_{\Lambda(\Phi_A)}^\sigma.\end{eqnarray} A simple calculation shows that \begin{equation}\label{E:p26}\varphi_A(st)\leq\varphi_A(s)\max(1,t).\end{equation} Note that $\varphi_{L^1\cap L^\infty}:t\mapsto\max(1,t)$ is the fundamental function of the r.\@i.\@BFS $L^1\cap L^\infty$. To see this observe that for any measurable set $E$ with $|E|=t$, we have by definition that $\varphi_{1\cap\infty}(t) = \|\chi_E\|_{1\cap\infty} = \max(\|\chi_E\|_\infty, \|\chi_E\|_1)=\max(1,t)$. The Young's function $$\Phi_{L^1\cap L^\infty}(t)=\left\{\begin{array}{ccc}t&\mbox{if}&t\leq 1\\\infty&\mbox{if}&t>1\end{array}\right.$$ corresponds to $\varphi_{L^1\cap L^\infty}$ according to the formula $\varphi_{L^1\cap L^\infty}(t)=1/\Phi_{L^1\cap L^\infty}^{-1}(1/t)$. The inequality (\ref{E:p26}) can therefore be written as $$\Phi^{-1}_A(st)\geq\Phi^{-1}_A(s)\Phi_{L^1\cap L^\infty}^{-1}(t).$$ Then an application of \cite[Theorem 8.15]{on}, combined with Lemma \ref{L:skewequi}, shows us that \begin{eqnarray*}\|F_{KU^{-1}}\|_{\Lambda(\Phi_A)}&=&\|f\otimes\chi_{KU^{-1}}\|_{\Lambda(\Phi_a)}\\&\leq&\|\chi_{KU^{-1}}\|_{L^1\cap L^\infty}\|f\|_{\Lambda(\Phi_A)}\\&=&\max(1,|KU^{-1}|)\|f\|_{\Lambda(\Phi_A)},\end{eqnarray*} and so 
\begin{eqnarray*}\frac{1}{|A|}\int_A|T^\#f|^\sigma(\omega) d\mu&\leq&\frac{2 (6\theta c)^\sigma}{1-\sigma}[\varphi_Y(|K|)\varphi_B(|A|)]^{-\sigma}\max(1,|KU^{-1}|^\sigma)\|f\|_{\Lambda(\Phi_A)}^\sigma\\&=&\frac{c_0^\sigma}{1-\sigma}[\varphi^*_B(|A|)]^\sigma|A|^{-\sigma}\|f\|_{\Lambda(\Phi_A)}^\sigma,\end{eqnarray*} where $c_0=2^{1/\sigma}6\theta c\varphi_Y(|K|)^{-1}\max(1,|KU^{-1}|).$

Therefore \begin{eqnarray*}\int_A|T^\#f|^\sigma(\omega)~d\mu&\leq&\frac{c_0^\sigma}{1-\sigma}[\varphi^*_B(|A|)]^\sigma|A|^{1-\sigma}\|f\|_{\Lambda(\Phi_A)}^\sigma,\end{eqnarray*} and so by Theorem \ref{thKol}, $T^\#$ is of $\Lambda$-weak type $(\varphi_A,\varphi_B)$.
\end{proof}

We remark again that in contrast to Theorem \ref{C:endcal}, Corollary \ref{C:CalderonTh1corr2} applies only to single transferable operators, not sequences. The obstacle in this regard is that the norm estimate obtained, namely $c_0$, depends on the set $U$, for which the measure may grow without bound in the case of sequences. However under some mild restrictions on the fundamental functions of the associated spaces, one can obtain results applicable to sequences. We start this phase of the analysis by defining some sets that shall be used extensively in the sequel.

\begin{definition}\label{D:calderonset}Let $T$ be a transferable operator and $U$ an open neighbourhood of $1$ as specified in Definition  \ref{D:to}(2). For $f\in L^{1+\infty}(\Omega)$, let $K$ be a measurable subset of $G$ and $\lambda>0$. Using the notation of Subsection \ref{SS:sal}, we define the following sets:
\begin{eqnarray*}E&=&\{\omega:|F'(\omega,1)|>\lambda\}\subset\Omega\\\overline{E}&=&\{(t,\omega):|F'_{KU^{-1}}(\omega,t)|>\lambda\}\in \Omega\times G.\end{eqnarray*} Also, for each $t\in G$, we define $$\overline{E}^t=\{\omega:|F'_{KU^{-1}}(\omega,t)|>\lambda\}\subset\Omega$$
and for each $\omega\in\Omega$, $$\overline{E}_\omega:=\{t:|F'_{KU^{-1}}(\omega,t)|>\lambda\}\subset G.$$ \end{definition}

We saw in the proof of Corollary \ref{cgconcrete} that we may view $F'$ and $F'_{\widetilde{K}U^{-1}}$ as members of $LK^{\rm r-loc}(\Omega\times G)$. Hence these sets are all measurable, and so by Fubini's theorem, $$|\overline{E}|=\int_\Omega |\overline{E}_\omega| d\mu(\omega)=\int_G |\overline{E}^t|~dt.$$

The next lemma is the key technical ingredient. It is in this Lemma that the semilocality and equi\-measura\-bility-preserving properties of transferable operators are used.

\begin{lemma}\label{L:end1}Let $(\Omega,\mu,G,\alpha)$ be a dynamical system with $(\Omega,\mu)$ resonant, and let $T$ be a transferable operator of weak type $(X,Y)$ with norm majorised by $c$.

Let $f\in X$, and $\lambda>0$, with $E$, $\overline{E}$, $\overline{E}^t$ and $\overline{E}_\omega$ the sets given in Definition \ref{D:calderonset}. Then we have 
\begin{eqnarray}\label{E:end2}|\overline{E}|&\geq&|K||E|\\\varphi_Y(|\overline{E}_\omega|)\label{E:eend1}&\leq&\frac{c}{\lambda}\|(F_{KU^{-1}})_\omega\|_X.\end{eqnarray}
\end{lemma}

\begin{proof}
For all $t\in K$ and almost all $\omega\in\Omega$, we have by Lemmas \ref{L:Kdominance} and \ref{L:equi} that  $$|F'_{KU^{-1}}(\omega,t)|\geq |F'(\omega,t)|=|F'(\alpha_t(\omega),1)|.$$ Therefore 
\begin{eqnarray*}\overline{E}^t&=&\{\omega:|F'_{KU^{-1}}(\omega,t)|>\lambda\}\\&\supseteq&\{\omega:|F'(\omega,t)|>\lambda\}\\&=&\{\omega:|F'(\alpha_t(\omega),1)|>\lambda\}\\&=&\alpha_{t^{-1}}\left(\{\omega:|F'(\omega,1)|>\lambda\}\right).\end{eqnarray*} 
This implies that $|\overline{E}^t|\geq |E|$, and hence that $$|\overline{E}|=\int_G|\overline{E}^t|~dt\geq\int_K|\overline{E}^t|~dt\geq \int_K|E|~dt=|K||E|.$$

By Lemma \ref{lem2p22} and the hypothesis on the weak type of $T$, we moreover have that 
$$\lambda\varphi_Y(|\overline{E}_\omega|) \leq \|T(F_{KU^{-1}})_\omega\|_{M^*(\varphi_Y)} \leq c\|(F_{KU^{-1}})_\omega\|_{X}
$$for all $\lambda>0$.\end{proof}

The following lemma is an almost immediate consequence of inequality (\ref{E:end2}) above.

\begin{lemma}\label{convlem1}
Let $T$ be a transferable operator, $K$ be a compact neighbourhood of $1\in G$ and $U$ the neighbourhood specified by the definition of semilocality. Given $f\in L^{1+\infty}(M)$, we have that 
$(\widetilde{T}(F_{KU^{-1}}))^*(t) \geq (T^\#(f))^*(t/|K|)$ for each $t\geq 0$.
\end{lemma}

\begin{theorem}\label{P:zyzz} Let $(T_n)$ be a sequence of operators given by \begin{equation}\label{E:TransferEg}T_n(f)=k_n*f,\end{equation} where $k_n\in L^1(G)$ is bounded and has bounded support. Suppose there are Young's functions $\Phi_A$, $\Phi_B, \Phi_C,\Phi_D$ and $\Phi_E$ with associated fundamental functions $\varphi_A,\ldots,\varphi_E$ satisfying \begin{eqnarray*}\varphi_C(t)\varphi_B^*(s)&\leq&\theta_1\varphi_A(st)\\\varphi_A(st)&\leq&\theta_2\varphi_D(s)\varphi_E(t)\end{eqnarray*} for all $s,t>0$.

Suppose further that there are measurable functions $\ell_0$ and $\ell_1$ on $G$ such that $\sup|k_n(s)|=\ell_0(s)\ell_1(s)$, $\ell_0\in L(\Phi_B)$ and $\ell_1\in L(\Phi_E)$. Then the operator  $T^\#$ defined by $T^\#f(\omega)=\sup_{n\in\nn}|T^\#_nf|(\omega)$ is a sublinear operator mapping $L(\Phi_D)$ into $L(\Phi_C)$.
\end{theorem}

\begin{proof}For each $N\in\nn$, let $T_Nf(\omega):=\max_{1\leq n\leq N}|T^\#_nf|(\omega)$. Then $(T_N)$ is a nondecreasing sequence of transferable operators. Clearly $T^\#=\sup_N T^\#_N$. By \cite[Theorem 8.18]{on}, 
\begin{eqnarray*}\|T_Nf\|_{L(\Phi_C)}&\leq &\|\max_{1\leq n\leq N}\big|\int_Gk_n(s)f(\alpha_{s^{-1}}(\omega))~ds\big|\|_{L(\Phi_C)}\\
&\leq &\|\int_G\max_{1\leq n\leq N}|k_n(s)||f(\alpha_{s^{-1}}(\omega))|ds\|_{L(\Phi_C)}\\
&\leq &\theta\|\ell_1\|_{L(\Phi_B)}\|\ell_0\otimes_{\alpha}f\|_{L(\Phi_A)}\\
&=&\theta_1\|\ell_1\|_{L(\Phi_B)}\|\ell_0\otimes f\|_{L(\Phi_A)} \rm{(by~ Lemma ~\ref{L:skewequi})}\\&\leq &\theta_1\theta_2\|\ell_1\|_{L(\Phi_B)}\|\ell_0\|_{L(\Phi_E)}\|f\|_{L(\Phi_D)}\end{eqnarray*} where the final inequality follows from \cite[Theorem 8.15]{on}.
\end{proof}

Recall the notation introduced after Definition \ref{D:LMOweaktype}, namely that for a fundamental function $\varphi$, by $\Lambda(\varphi;\Omega)$ we mean the space $\Lambda(\varphi)$ of functions on $\Omega$, and similarly for $L(\varphi;\Omega)$ and $M(\varphi;\Omega)$.

\begin{theorem}\label{C:endcal9}Let $(\Omega,\mu,G,\alpha)$ be a dynamical system where the group $G$ is amenable. Suppose that we are given fundamental functions $\varphi_A$, $\varphi_W$, $\varphi_X$, $\varphi_Y$ and $\varphi_Z$ and positive constants $\theta_1$, $\theta_2$ satisfying 
$$\varphi_A(t)\varphi_X(s)\leq\varphi_Y(\theta_1st), \mbox{ and } \varphi_Z(st)\leq\theta_2\varphi_A(s)\varphi_W(t)\mbox{ for all }t>0,s>0,$$with $\varphi_Z\in\mathcal{U}$, $\lim_{t\to\infty}\varphi_A(t)=\infty$ and $\lim_{t\to 0}\varphi_A(t)=0$.
Let $(T_n)$ be a sequence of transferable operators of weak type $(X,Y)$, and in addition suppose that  $T:=\sup_n|T_n|$ is of weak type $(X,Y)$. Then the operator $T^\#:=\sup_n |T^\#_n|$ is of $\Lambda$-weak type $(\varphi_W,\varphi_Z)$.\end{theorem}

\begin{proof}As in Theorem \ref{P:zyzz} we may assume that the sequence of operators $|T_n|$ is nondecreasing, by passing to the sequence $T_N:=\sup_{n\leq N}|T_n|$. Clearly $T^\#=\sup_NT^\#N$.

Let $c$ be the bound of the operator $T$. For any $n\in\mathbb{N}$ and any $g\in X$ we will then have that $\|T(g)\|_{M^*(\varphi_Y)}\leq c \|g\|_X$. Let $n\in\mathbb{N}$ be given and let $E$, $\overline{E}$, $\overline{E}^t$ and $\overline{E}_\omega$ be the sets specified in Definition \ref{D:calderonset}, for the operator $T_n$, a compact neighbourhood $K\subset G$ of $1\in G$, and the function $f=\chi_A$, where $A\subset\Omega$ is measurable set with finite measure. Let $U$ be the set specified in part (2) of Definition \ref{D:to}.

It is easy to verify that $(\chi_A\otimes_\alpha\chi_{KU^{-1}})(\omega,t)= \chi_{KU^{-1}}(t)\chi_A(\alpha_t(\omega))$ is a characteristic function on $\Omega\times G$. So for some measurable subset $\widetilde{E}\subset \Omega\times G$, we have that $\chi_{\widetilde{E}}= \chi_A\otimes_\alpha\chi_{KU^{-1}}$. So in this case $(F_{KU^{-1}})_\omega =(\chi_A\otimes_\alpha\chi_{KU^{-1}})_\omega =(\chi_{\widetilde{E}})_\omega=\chi_{\widetilde{E}_\omega}$. By Lemma \ref{L:skewequi}, we have moreover that 
\begin{eqnarray*}
|\widetilde{E}|&=&  \int_{\Omega\times G}\chi_A\otimes_\alpha\chi_{KU^{-1}}\,(d\mu(\omega)\times dt)\\
&=& \int_{\Omega\times G}\chi_A\otimes\chi_{KU^{-1}}\,(d\mu(\omega)\times dt)\\
&=& \int_G \int_\Omega\chi_{KU^{-1}}(t)\chi_A(\omega)\,d\omega dt\\
&=&|KU^{-1}||A|. 
\end{eqnarray*}

From inequality (\ref{E:eend1}), we therefore have that
$$\varphi_Y(|\overline{E}_\omega|)\leq\frac{c}{\lambda}\|(F_{KU^{-1}})_\omega)\|_X  = 
\frac{c}{\lambda}\|\chi_{\widetilde{E}_\omega}\|_X=\frac{c}{\lambda}\varphi_X(|\widetilde{E}_\omega|).$$

Observe that since $\varphi_Y$ is nondecreasing, we may assume without loss of generality that $\theta_1\geq 1$ in the hypothesis of the theorem. Assume for the moment that $\varphi_A$ and $\varphi_Y$ each are invertible in the usual sense. It then follows from the above inequality and the inequality $\varphi_A(t)\varphi_X(s)\leq\varphi_Y(\theta_1st)$, that 
$$\varphi_Y(|\overline{E}_\omega|)\leq \varphi_A(\varphi_A^{-1}(\frac{c}{\lambda}))\varphi_X(|\widetilde{E}_\omega|)\leq \varphi_Y(\theta_1\varphi_A^{-1}(\frac{c}{\lambda})|\widetilde{E}_\omega|),$$ and hence that
$$|\overline{E}_\omega|\leq \theta_1\varphi_A^{-1}(\frac{c}{\lambda})|\widetilde{E}_\omega|.$$Integrating over $\Omega$ now yields the conclusion that 
\begin{eqnarray}\nonumber|K||E|&\leq&|\overline{E}|=\int_\Omega |\overline{E}_\omega|\,d\mu \leq \theta_1\varphi_A^{-1}(\frac{c}{\lambda})\int_\Omega |\widetilde{E}_\omega|\,d\mu=\theta_1\varphi_A^{-1}(\frac{c}{\lambda})|\widetilde{E}|\\\label{E:prevcent}&=&\theta_1\varphi_A^{-1}(\frac{c}{\lambda})|KU^{-1}||A|.\end{eqnarray}

On the one hand, $|KU^{-1}|/|K|\geq 1$. On the other hand, by definition of $U$, $\overline{U}$ is compact and so by \cite[\S 1]{emgr},(cf. \cite[Corollary 4.14]{pat}), 
$$\inf\left\{\frac{|\overline{U}K|}{|K|}:K\in\mathscr{C}({G}), |K|>0\right\}\leq 1,$$where $\mathscr{C}(G)$ is the collection of all compact subsets of $G$. Consequently, the fact that Haar measure is preserved under taking inverses yields that $$1\leq\frac{|KU^{-1}|}{|K|}=\frac{|UK^{-1}|}{|K^{-1}|}\leq\frac{|\overline{U}K^{-1}|}{|K^{-1}|}$$and so $$\inf\left\{\frac{|KU^{-1}|}{|K|}:K\in\mathscr{C}(G), |K|>0\right\}=1.$$

The inequality (\ref{E:prevcent}) therefore clearly yields the conclusion that $$|E|\leq\theta_1\varphi_A^{-1}(\frac{c}{\lambda})|A|.$$But then 
$$\varphi_Z(|E|)\leq\varphi_Z(\theta_1\varphi_A^{-1}(\frac{c}{\lambda})|A|)\leq \theta_2\varphi_A(\varphi_A^{-1}(\frac{c}{\lambda}))\varphi_W(\theta_1|A|)\leq(\theta_1+1)\frac{\theta_2 c}{\lambda}\varphi_W(|A|).$$
(In the last inequality we used the fact that $\varphi_W(\theta_1u)\leq (\theta_1+1)\varphi_W(u)$ for all $u>0$, which in turn follows fairly directly from the facts that $\varphi_W$ is nondecreasing and $\frac{\varphi_W(s)}{s}$ nonincreasing.) Keeping in mind that $|E|=m(T^\#_Nf,\lambda)$, it is clear that we have proved that for all $\lambda>0$,  $$\lambda\varphi_Z(m(T^\#_N(\chi_A),\lambda))\leq (\theta_1+1)\theta_2 c\varphi_W(|A|).$$

We proceed to deal with the case where possibly $\varphi_A$ and $\varphi_Y$ are constant on some part of $[0,\infty)$. Given $1>\epsilon >0$, we replace $\varphi_A$ with the function $\varphi_{A,\epsilon}$ defined by $\varphi_{A,\epsilon}(t)=\epsilon^{1/(1+t)}\varphi_A(t)$ for all $t\geq 0$. The function $t\to \epsilon^{1/(1+t)}$ is a strictly increasing function mapping $[0, \infty)$ onto $[\epsilon, 1)$. Thus since by hypothesis  
$\lim_{t\to\infty}\varphi_A(t)=\infty$ and $\lim_{t\to 0}\varphi_A(t)=0$,  $\varphi_{A,\epsilon}$ is then a strictly increasing function mapping $[0,\infty)$ onto $[0,\infty)$. In addition for all $t>0$ we have that $\epsilon\varphi_{A}(t)\leq \varphi_{A,\epsilon}(t)\leq\varphi_{A}(t)$. Given $t>0$, the value $\varphi_{A,\epsilon}(t)$ then clearly increases  to $\varphi_{A}(t)$ as $\epsilon$ increases to 1. The functions $\varphi_{X,\epsilon}$ and $\varphi_{Y,\epsilon}$, are similarly defined. Observe that since by assumption $\theta_1\geq 1$ it now follows that $\frac{1}{1+s}+\frac{1}{1+t}>\frac{1}{1+st}\geq\frac{1}{1+\theta_1st}$ for all $s,t>0$. Given that  $1>\epsilon>0$, it is now clear that $\epsilon^{1/(1+t)} \epsilon^{1/(1+s)} < \epsilon^{1/(1+st)}\leq \epsilon^\frac{1}{1+\theta_1st}$. Thus the inequalities 

$$\varphi_A(t)\varphi_X(s)\leq\varphi_Y(\theta_1st), \mbox{ and } \varphi_Z(st)\leq\theta_2\varphi_A(s)\varphi_W(t)\mbox{ for all }t>0,s>0$$
imply the inequalities 
$$\varphi_{A,\epsilon}(t)\varphi_{X,\epsilon}(s)\leq\varphi_{Y,\epsilon}(\theta_1st)\mbox{ and } \varphi_Z(st)\leq\frac{\theta_2}{\epsilon}\varphi_{A,\epsilon}(s)\varphi_W(t)
\mbox{ for all }s>0, t>0.$$Since in addition
$$\varphi_{Y,\epsilon}(|\overline{E}_\omega|) < \varphi_{Y}(|\overline{E}_\omega|)
\leq\frac{c}{\lambda}\varphi_X(|\widetilde{E}_\omega|)< \frac{c}{\epsilon\lambda}\varphi_{X,\epsilon}(|\widetilde{E}_\omega|),$$we may now argue as before to obtain the conclusion that
$$\lambda\varphi_Z(m(T^\#_N(\chi_A),\lambda))\leq (\theta_1+1)\theta_2 \frac{c}{\epsilon}\varphi_W(|A|).$$On letting $\epsilon$ increase to 1, we have as before that
$$\lambda\varphi_Z(m(T^\#_N(\chi_A),\lambda))\leq (\theta_1+1)\theta_2 c\varphi_W(|A|).$$
Notice that since $\Lambda(\varphi_W;\Omega)$ contractively embeds into $L^{1+\infty}(\Omega)$, it clearly follows from Lemma \ref{L:F} and the diagram in Definition \ref{R:transfer}, that $T^\#_N$ maps sequences converging in $\Lambda(\varphi_W;\Omega)$ onto sequences converging in measure on subsets of $\Omega$ of finite measure. By Lemma \ref{lem2p22}, we have shown that
$$\sup_{t>0}\varphi_Z(t)(T^\#_N(\chi_A))^*(t)=\sup_{\lambda>0}\lambda\varphi_Z(m(T^\#_N(\chi_A),\lambda))\leq (\theta_1+1)\theta_2 c\varphi_W(|A|)$$ for all measurable $A\subset\Omega$. As $(T^\#_N)$ is nondecreasing and $T^\#=\sup_NT^\#_N$, $$\sup_{t>0}\varphi_Z(t)(T^\#(\chi_A))^*(t)\leq (\theta_1+1)\theta_2 c\varphi_W(|A|),$$ 
and so by \cite[Theorem 2.5]{shlax} the operator $T^\#$ is of $\Lambda$-weak type $(\varphi_W,\varphi_Z)$.
\end{proof}

The above theorem of course applies in particular to operators with weak type $(p,p)$. However it does not guarantee a transfer of $L$-weak type. To obtain such a result, additional restrictions need to be placed on the fundamental functions involved, as is demonstrated by the next two results.

\begin{theorem}\label{C:endcal}Let $(\Omega,\mu,G,\alpha)$ be a dynamical system where the group $G$ is amenable. Suppose that we are given fundamental functions $\varphi_A$, $\varphi_W$, $\varphi_X$, $\varphi_Y$ and $\varphi_Z$ and positive constants $\theta_1$, $\theta_2$ satisfying 
$$\varphi_A(t)\varphi_X(s)\leq\varphi_Y(\theta_1st), \mbox{ and } \varphi_Z(st)\leq\theta_2\varphi_A(s)\varphi_W(t)\mbox{ for all }t>0,s>0$$with  $\lim_{t\to\infty}\varphi_A(t)=\infty$ and $\lim_{t\to 0}\varphi_A(t)=\lim_{t\to 0}\varphi_W(t)=0$. Suppose also that $\varphi_X$ has the associated Young's function $\Phi_X$ belonging to the class $\Delta'$. Let $(T_n)$ be a sequence of transferable operators of $L$-weak type $(\varphi_X,\varphi_Y)$ and suppose that in addition $T:=\sup_n|T_n|$ is of $L$-weak type $(\varphi_X,\varphi_Y)$. Then the operator $T^\#:=\sup_n |T^\#_n|$ is of $L$-weak type $(\varphi_X,\varphi_Z)$.\end{theorem}

\begin{proof}Let $c$ be the bound of the operator $T$. For any $n\in\mathbb{N}$ and any $g\in L(\varphi_X;G)$ we have that $\|T(g)\|_{M^*(\varphi_Y)}\leq c \|g\|_{L(\varphi_X;G)}$. Using the fact that $\Phi_X$ satisfies the $\Delta'$ (and hence also the $\Delta_2$) condition, it is an easy exercise to see that $\Phi_X(t)>0$ for all $t>0$, and hence that $\Phi_X$ has a `proper' inverse. For such Young's functions it is also known that $\int_G\Phi_X(|g|(t))dt<\infty$ for all $g\in L(\varphi_X;G)$, and similarly that $\int_G\Phi_X(|f|(\omega))d\mu(\omega)<\infty$ for all $f\in L(\varphi_X;\Omega)$. Using the fact that $\Phi_X(st)\leq c_0\Phi_X(s)\Phi_X(t)$ for all $s$ and $t$, it is easy to see that for any $g\in L(\varphi_X;G)$, we have that 
\begin{eqnarray*}
&&\int_G \Phi_X\left(|g|(s)/[\varphi_X(c_0\int_G \Phi_X(|g|)(t)\,dt)]\right)\,ds\\
&\leq& c_0 \Phi_X\left(1/[\varphi_X(c_0\int_G \Phi_X(|g|)(t)\,dt)]\right)\int_G\Phi_X(|g|(s))\,ds\\
&=& c_0\Phi_X\left(\Phi_X^{-1}(1/[c_0\int_G \Phi_X(|g|)(t)\,dt])\right)\int_G\Phi_X(|g|(s))\,ds\\
&=&1.
\end{eqnarray*}
It follows that 
$$\|g\|_X\leq \varphi_X\left(c_0\int_G \Phi_X(|g|)(t)\,dt\right)\mbox{ for all }g\in L(\varphi_X;G).$$Now let $E$, $\overline{E}$, $\overline{E}^t$ and 
$\overline{E}_\omega$ be the sets specified in Definition \ref{D:calderonset}, for the operator $T_n$ and the function $f\in L(\varphi_X;\Omega)$. Let 
$U$ be the set specified in Definition \ref{D:to}(2). For some compact neighbourhood $K\subset G$ of $1\in G$, we then have by Lemma \ref{L:end1} that
\begin{eqnarray*}\varphi_Y(|\overline{E}_\omega|)\leq&\frac{c}{\lambda}\|(F_{KU^{-1}})_\omega\|_X \leq \frac{c}{\lambda}\varphi_X(c_0\int_G \Phi_X(|(F_{KU^{-1}})_\omega|)(t)\,dt). \end{eqnarray*}Arguing as in the proof of Theorem \ref{C:endcal9}, it is clear that we may assume that both $\varphi_A$ and $\varphi_Y$ have `proper' inverses, which we may then use to conclude from the above that
$$\varphi_Y(|\overline{E}_\omega|)\leq 
\varphi_Y\left(\theta_1c_0\varphi_A^{-1}\left(\frac{c}{\lambda}\right)\int_G \Phi_X(|(F_{KU^{-1}})_\omega|)(t)\,dt\right),$$and hence that 
\begin{equation}\label{eq:xxx}|\overline{E}_\omega|\leq \theta_1c_0\varphi_A^{-1}\left(\frac{c}{\lambda}\right)\int_G 
\Phi_X(|(F_{KU^{-1}})_\omega|)(t)\,dt. \end{equation}
Next observe that since $\Phi_X$ is zero at zero, it follows that 
$$\Phi_X(|F_{KU^{-1}}|)(\omega,t) = \Phi_X(\chi_{KU^{-1}}(t)|f|(\alpha_t(\omega))) = \chi_{KU^{-1}}(t)\Phi_X(|f|(\alpha_t(\omega)))$$for all $t\in G$ and all $\omega\in \Omega$. Recall that by assumption we also have that $\Phi_X(|f|)\in L^1(\Omega)$. So if we integrate $\int_G \Phi_X(|(F_{KU^{-1}})_\omega|)(t)\,dt$ over $\Omega$, we may use Lemma \ref{L:skewequi} to conclude that
\begin{eqnarray*}\int_\Omega\int_G \Phi_X(|F_{KU^{-1}}|)(\omega,t)\,dt\,d\mu(\omega) &=& \int_\Omega\int_G (\Phi_X(|f|)\otimes_\alpha\chi_{KU^{-1}})(\omega,t)\,dt\,d\mu(\omega)\\
&=& \int_\Omega\int_G (\Phi_X(|f|)\otimes\chi_{KU^{-1}})(\omega,t)\,dt\,d\mu(\omega)\\
&=&|KU^{-1}|\int_\Omega\Phi_X(|f|)d\mu(\omega).
\end{eqnarray*}
Integrating equation (\ref{eq:xxx}) over $\Omega$ therefore yields
$$|K||E|\leq|\overline{E}|\leq \theta_1c_0\varphi_A^{-1}\left(\frac{c}{\lambda}\right)|KU^{-1}|\int_\Omega\Phi_X(|f|)d\mu(\omega).$$As in the proof of Theorem \ref{C:endcal9}, $\displaystyle\inf\{|KU^{-1}|/|K|:K\in\mathscr{C}(G), |K|>0\}=1$. Hence the above inequality yields the fact that 
$$|E|\leq \theta_1c_0\varphi_A^{-1}\left(\frac{c}{\lambda}\right)\int_\Omega\Phi_X(|f|)d\mu(\omega).$$Once again arguing as in the proof of Theorem \ref{C:endcal9}, we may now conclude from this inequality that  
$$\lambda\varphi_Z(|E|)\leq \theta_2 c \varphi_W(\theta_1c_0\int_\Omega\Phi_X(|f|)d\mu(\omega)).$$Taking the supremum over $\lambda>0$ now yields
$$\|T_n^\#(f)\|_{M^*(\varphi_Z)}=\sup_{\lambda>0}\lambda\varphi_Z(m(T^\#_n(\chi_A),\lambda))\leq \theta_2 c \varphi_W(\theta_1c_0\int_\Omega\Phi_X(|f|)d\mu(\omega)).$$This inequality ensures that $T_n^\#(f_m)$ converges to $T_n^\#(f)$ whenever $\int_\Omega\Phi_X(|f_m-f|)d\mu(\omega)\to 0$ as $m\to \infty$. But since $\Phi_X\in \Delta_2$, it follows from 
\cite[Theorem 3.4.12]{raore1} that the convergence $\int_\Omega\Phi_X(|f_m-f|)d\mu(\omega)\to 0$ is equivalent to norm convergence of $(f_m)$ to $f$ in $L(\varphi_X;\Omega)$. It  follows that $T^\#_n$ is a bounded operator from $L(\varphi_X;\Omega)$ to $M^*(\varphi_Z)$. 

We proceed to estimate the norm of $T_n^\#$. For any $f\in L(\varphi_X;\Omega)$ with $\|f\|_{L(\varphi_X;\Omega)}\leq 1$, we have that $\int_\Omega\Phi_X(|f|)d\mu(\omega)\leq\|f\|_{L(\varphi_X;\Omega)}\leq 1$ 
by \cite[Lemma 4.8.8]{besh}. For such $f$ we estimate
$$\|T_n^\#(f)\|_{M^*(\varphi_Z)}\leq \theta_1 c \varphi_W(\theta_1c_0\int_\Omega\Phi_X(|f|)d\mu(\omega)) \leq \theta_2 c \varphi_W(\theta_1c_0).$$This in turn leads to the conclusion that $\|T_n^\#\|\leq \theta_2 ce \varphi_W(\theta_1c_0)$ for every $n\in \mathbb{N}$. Reasoning as in the proof of Theorem \ref{C:endcal9}, it is clear that we may suppose that $(T^\#_n)$ is a nondecreasing sequence. In view of this fact, 
the operator $T^\#:=\sup_n |T^\#_n|$ is of $L$-weak type $(\varphi_W,\varphi_Z)$ with norm $\|T^\#\|\leq \theta_2c \varphi_W(\theta_1c_0)$.
\end{proof}

If in the above theorem one has an a priori weak type condition formulated in terms of modulars, the restrictions on the fundamental functions can be significantly weakened.

\begin{theorem}\label{C:endcal2}Let $(\Omega,\mu,G,\alpha)$ be a dynamical system where the group $G$ is amenable. Suppose that we are given fundamental functions $\varphi_A$, $\varphi_W$, $\varphi_X$, $\varphi_Y$ and $\varphi_Z$ and positive constants $\theta_1$, $\theta_2$ satisfying 
$$\varphi_A(t)\varphi_X(s)\leq\varphi_Y(\theta_1st), \mbox{ and } \varphi_Z(st)\leq\theta_2\varphi_A(s)\varphi_W(t)\mbox{ for all }t>0,s>0$$with  $\lim_{t\to\infty}\varphi_A(t)=\infty$ and $\lim_{t\to 0}\varphi_A(t)=\lim_{t\to 0}\varphi_W(t)=0$. Suppose also that $\varphi_X$ has associated Young's function $\Phi_X$, and
\begin{itemize}
\item either $\Phi_X$ belongs to the class $\Delta_2$ globally, 
\item or $\varphi_W\in \mathcal{U}$. 
\end{itemize}
Then for any sequence $(T_n)$ of transferable operators satisfying the inequality
$$\|T_ng\|_{M^*(\varphi_Y)}\leq c_1\varphi_X\left(c_0\int_G \Phi_X(|g|)(t)\,dt\right)\mbox{ for all }g\in L(\varphi_X;G)\mbox{ and all }n\in \mathbb{N},$$
the operator $T^\#:=\sup_n |T^\#_n|$ will be of $L$-weak type $(\varphi_X,\varphi_Z)$.\end{theorem}

\begin{proof} Let $E$, $\overline{E}$, $\overline{E}^t$ and 
$\overline{E}_\omega$ be the sets specified in Definition \ref{D:calderonset}, for the operator $T_n$ and the function $f\in L(\varphi_X;\Omega)$. Let 
$U$ be the set specified in Definition \ref{D:to}(2). For some compact neighbourhood $K\subset G$ of $1\in G$, it then follows from Lemma \ref{lem2p22} and the given modular inequality that
\begin{eqnarray*}\varphi_Y(|\overline{E}_\omega|)\leq \frac{c_1}{\lambda}\varphi_X(c_0\int_G \Phi_X(|(F_{KU^{-1}})_\omega|)(t)\,dt) \end{eqnarray*}
With $c$ replaced by $c_1$ this part of the proof now proceeds exactly as in Theorem \ref{C:endcal} to the point where we obtain the conclusion that
$$\|T_n^\#(f)\|_{M^*(\varphi_Z)}=\sup_{\lambda>0}\lambda\varphi_Z(m(T^\#_n(\chi_A),\lambda))\leq \theta_2 c_1 \varphi_W(\theta_1c_0\int_\Omega\Phi_X(|f|)d\mu(\omega)).$$
In concluding the proof we consider two cases:

\bigskip

\textbf{Case 1}($\Phi_X\in \Delta_2$): In this case we may argue exactly as in the final part of the proof of Theorem \ref{C:endcal}, to conclude from the above inequality that the operator $T^\#:=\sup_n |T^\#_n|$ is of $L$-weak type $(\varphi_W,\varphi_Z)$ with norm $\|T^\#\|\leq \theta_2c_1 \varphi_W(\theta_1c_0)$.

\bigskip

\textbf{Case 2}($\varphi_Z\in \mathcal{U}$): In this case it follows from \cite[Theorem 2.2]{shlax} that the space $M^*(\varphi_Z)$ is isomorphic to the Banach space $M(\varphi_Z)$ with norm $\|g\|_{M(\varphi_Z)}=\sup_{t>0}g^{**}(t)\varphi_Z(t)$. Hence the inequality can then be reformulated as the claim that 
$$\|T_n^\#(f)\|_{M(\varphi_Z)} \leq c_2 \varphi_W(\theta_1c_0\int_\Omega\Phi_X(|f|)d\mu(\omega))\mbox{ for all }f\in L(\varphi_X;\Omega)$$for some constant $c_2>0$ and all $n\in \mathbb{N}$. Let $n$ be given. In the language of \cite{mus}, both of the quantities $\rho(f)=\int_\Omega\Phi_X(|f|)d\mu(\omega)$ and $\sigma(g)=\|g\|_{M(\varphi_Z)}$ are convex modulars. Let $A\subset L(\varphi_X;\Omega)$ be a set for which there exist positive constants $M$ and $a$ with $\int_\Omega\Phi_X(a|f|)d\mu(\omega)\leq M$ for every $f\in A$. By the above inequality we will then have that 
$a\|T_n^\#(f)\|_{M(\varphi_Z)}\leq c_2 \varphi_W(\theta_1c_0M)$ for all $f\in A$. So by \cite[Theorem 5.5 \& Definition 5.9]{mus} $T_n^\#$ is a so called $(\rho,\sigma)$-bounded map. Hence \cite[Theorem 5.10]{mus} is applicable. If we let $A=\{\frac{f}{\|f\|_{\Lambda(\varphi_X;\Omega)}} : f\in L(\varphi_X;\Omega)\}$, a perusal of the proof of \cite[Theorem 5.10]{mus} reveals that then
$a\|T_n^\#(f)\|_{M(\varphi_Z)}=\sigma(aT_n^\#(f))\leq c_2 \varphi_W(\theta_1c_0M)\|f\|_{\Lambda(\varphi_X;\Omega)}$ for all $f\in \Lambda(\varphi_X;\Omega)$. In other words 
$$\|T_n^\#(f)\|_{M(\varphi_Z)} \leq \frac{c_2}{a} \varphi_W(\theta_1c_0M)\|f\|_{\Lambda(\varphi_X;\Omega)}\mbox{ for all }f\in \Lambda(\varphi_X;\Omega).$$Finally notice that the constants $M$ and $a$ depend only on the set $A=\{\frac{f}{\|f\|_{\Lambda(\varphi_X;\Omega)}} : f\in \Lambda(\varphi_X;\Omega)\}$, and not on $n$. Hence the above norm estimate holds for every $n\in \mathbb{N}$. Using this fact, the same reasoning as was used in the proof of Theorem \ref{C:endcal9} to show that we may assume the sequence $(T^\#_n)$ to be nondecreasing, can now similarly be used to conclude that  
the operator $T^\#:=\sup_n |T^\#_n|$ is a bounded map from $\Lambda(\varphi_X;\Omega)$ into $M(\varphi_Z)$, with norm $\|T^\#\|\leq \frac{c_2}{a} \varphi_W(\theta_1c_0M)$.
\end{proof}

If in either of the previous two theorems we take all fundamental functions to be $t^{1/p}$ where $1<p<\infty$, we obtain the following easy corollary. Note that this recovers Calder\'on's result for transferable operators of weak type $(p,p)$. However it does go further than the result of Calder\'on, in that it is valid for actions of arbitrary amenable locally compact groups.

\begin{corr}Let $(\Omega,\mu,G,\alpha)$ be a dynamical system where $(\Omega,\mu)$ is resonant and the group $G$ is amenable. Let $(T_n)$ be a sequence of transferable operators of weak type $(p,p)$ where $1\leq p\leq\infty$, and $T:=\sup_n|T_n|$ is also of weak type $(p,p)$. Then the operator $T^\#:=\sup_n |T^\#_n|$ has the same weak type as $T$.\end{corr}

In closing this subsection, we show that for every sequence of transferable operators of weak type $(X,X)$, inequalities of the type required by Theorems \ref{C:endcal9}, \ref{C:endcal} and \ref{C:endcal2}, always pertain. In this endeavour we first recall the definition of the Boyd indices of a r.\@i.\@BFS $X$ which is defined by a function norm $\rho$. First, by the Luxemburg 
Representation Theorem \cite[Theorem 2.4.10]{besh} there is a (not necessarily unique) r.\@i.\@BFS $\overline{X}$ over the positive reals 
with Lebesgue measure, defined by a function norm $\overline{\rho}$ which is related to $\rho$ by the formula 
$$\overline{\rho}(f^*)=\rho(f)$$ for every $f\in X$. Now for each $t\in\rn^+$ we define the \textit{dilation operator} $E_t$ by 
$$(E_tg)(s)=g(st)$$ for all $s\in\rn^+$ and $g$ a measurable and finite a.e.\@ function on $[0,\infty)$. Let $h_X(t)$ denote the 
operator norm of $E_{1/t}$: that is, $h_X(t)=\|E_{1/t}\|_{\mathcal{B}(\overline{X})}$ for $t>0$. Define $h_X(0)=0$. In 
\cite[Section 3.5]{besh}, the authors thoroughly develop the basics of the theory, including the fact that $h_X$ is submultiplicative. 
Note that it is also quasiconcave, for by \cite[Proposition 3.5.11]{besh}, $h_X$ is nondecreasing and $h_X(t)/t=h_{X'}(1/t)$, which is 
nonincreasing. Furthermore, $h_X(t)>0$ for $t>0$, which is a consequence of the next Lemma. 

\begin{lemma}\label{L:newBoyd}If $X$ is a r.\@i.\@BFS then for all $s,t>0$, \begin{eqnarray*}\varphi_X(st)&\leq &h_X(t)\varphi_X(s)\\\varphi_X(st)&\geq&h^*_{X'}(t)\varphi_X(s).\end{eqnarray*}\end{lemma}

\begin{proof} Consider a function space $\overline{X}$ over the positive reals given by the Luxemburg Representation Theorem mentioned above and fix $s,t\in\rn^+$. Note that \begin{eqnarray*}\varphi_X(st)&=&\|\chi_{(0,st)}\|_{\overline{X}}\\&=&\|E_{1/t}\chi_{(0,s)}\|_{\overline{X}}\\&\leq&\|E_{1/t}\|_{\mathcal{B}(\overline{X})}\|\chi_{(0,s)}\|_{\overline{X}}\\&=&h_X(t)\varphi_X(s).\end{eqnarray*} From this inequality, we immediately deduce that $\varphi^*_X(st)\geq h^*_X(t)\varphi^*_X(s)$. As $\varphi^*_X=\varphi_X'$, this can be written as $\varphi_{X'}(st)\geq h^*_X(t)\varphi_{X'}(s)$. Equivalently, $\varphi_X(st) \geq h^*_{X'}(t)\varphi_X(s)$.\end{proof}

Note that as $h_X$ is submultiplicative, $h^*_{X'}$ is \textit{supermultiplicative}, in that $h^*_{X'}(st)\geq h^*_{X'}(s)h^*_{X'}(t)$. So we have shown that any fundamental function $\varphi_X$ can be bounded above by a submultiplicative and below by a supermultiplicative function (up to a constant factor), in that $$\varphi_X(1)h^*_{X'}(t)\leq\varphi_X(t)\leq \varphi_X(1)h_X(t).$$

The \textit{upper and lower Boyd indices} of $X$, denoted respectively by $\underline{\alpha}_X$ and $\overline{\alpha}_X$ are given by \begin{eqnarray*}\underline{\alpha}_X=\lim_{t\to 0^+}\frac{\ln h_X(t)}{\ln t},&&\overline{\alpha}_X=\lim_{t\to\infty}\frac{\ln h_X(t)}{\ln t}.\end{eqnarray*}

Recall from Definition \ref{D:LU} the definitions of the $\mathcal{L}$- and $\mathcal{U}$-indices of a fundamental function. 
 Note that the fundamental indices (defined at the end of Subsection 3.2) of a fundamental function are always defined, even if the $\mathcal{L}$- or $\mathcal{U}$- indices are not. The relation between the Boyd and fundamental indices given in the next Lemma, is well known - see for example \cite[Chapter 3, exercise 14]{besh} and \cite[Remark 1.4]{shlax}.

\begin{lemma}\label{L:indexComp}Let $X$ be a r.\@i.\@BFS with fundamental function $\varphi$. Then $$0\leq\underline{\alpha}_X\leq\underline{\beta}_X\leq\overline{\beta}_X\leq\overline{\alpha}_X\leq 1.$$
Moreover, $\varphi\in\mathcal{U}$ if and only if $\overline{\beta}_X<1$ and in this case $\overline{\beta}_X=\rho^\varphi_{\mathcal{U}}$, and $\varphi\in\mathcal{L}$ if and only if $\underline{\beta}_X>0$ and in this case $\underline{\beta}_X=\rho^\varphi_{\mathcal{L}}$.
\end{lemma}

\section{Applications of the Transfer Principle}\label{S:appl}

The four main results of the previous Section, Corollary \ref{C:CalderonTh1corr2} and Theorems \ref{C:endcal9}, \ref{C:endcal}  and \ref{C:endcal2}, are powerful enough to yield a great many maximal inequalities. We can of course use these maximal inequalities to derive a variety of pointwise convergence theorems. This is illustrated in the present Section by Theorem \ref{theoBirk} and Corollary \ref{corrInterpBirk}.

We turn to the derivation of pointwise ergodic theorems. Henceforth, in the dynamical system $(\Omega,\mu,G,\alpha)$, not only will $(\Omega,\mu)$ be $\sigma$-finite and resonant, but $G$ will be an \textit{abelian, additive, second countable} locally compact Hausdorff group with identity element $0$.
We shall work with transfer operators generated by sequences of convolution operators. So as in Theorem  \ref{P:zyzz}, we consider a sequence $(T_n)$ of operators on $L^{\rm loc}(G)$ given by \begin{equation}\label{E:p38}T_n(f)=k_n*f,\end{equation} where $k_n\in L^1(G)$ is bounded and has bounded support, and $f$ is locally integrable. We also define $Tf:=\sup_n |T_n(f)|$.
Using the Transfer Principle and given information about the functions $k_n$ and the space $X$, we show that the transferred operators $T^{\#}_n$ satisfy a pointwise convergence theorem: that is, $T^{\#}_nf(\omega)$ converges a.e.\@ for all $f\in X$ as $n$ tends to infinity.

To achieve this goal, our strategy is the following three step programme.\begin{enumerate} \item Given the weak type of the operator $T$, find the weak type of $T^{\#}$.\item In the domain of $T^{\#}$ computed in step (1), identify a dense subset $D$ for which the pointwise convergence of $(T^{\#}_nf)$ can be verified for all $f\in D$.\item Use an appropriate version of Banach's Principle to extend the a.e.\@ convergence of step (2) to the whole domain of $T^{\#}$.\end{enumerate}

To do step (1), we shall use results obtained earlier in this paper. For step (3), we prove the following variation on the theme of \cite[Corollary 4.5.8]{besh} and \cite[Theorem 1.1.1]{ga}, which provides the final link between maximal inequalities and pointwise ergodic theorems.

\begin{proposition}\label{P:osc}Let $X$ and $Y$ be r.i.\@ BFSs over a measure space $(\Omega,\mu)$. Let $(T_n)$ be a sequence of linear operators on $X$ and define the maximal operator $T$ by $T(f)=\sup_n|T_n(f)|.$ If\begin{enumerate}\item  there is a dense subset $D\subseteq X$ such that for all $f\in D$,  $(T_n(f)(\omega))$ converges for $\mu$-a.e.\@ $\omega\in\Omega$, \item $T$ is of weak-type $(X,Y)$,
\end{enumerate} then $(T_n(f)(\omega))$ converges for $\mu$-a.e.\@ $\omega\in\Omega$ and all $f\in X$.\end{proposition}

\begin{proof} Define the oscillation $\mathcal{O}_f$ of $f\in X$ as follows. For any $\omega\in\Omega$ set $$\mathcal{O}_f(\omega)=\limsup_{n,m\rightarrow\infty}|T_n(f)(\omega)-T_m(f)(\omega)|.$$ Clearly the linearity of the operators $T_n$ implies that $\mathcal{O}_f(\omega)\leq\mathcal{O}_g(\omega)+\mathcal{O}_{f-g}(\omega).$ 

For any $g\in D$ and $\delta>0$, we have $\mu(\{\omega:\mathcal{O}_g(\omega)>\delta\})=0$, due to the $\mu$-a.e.\@ convergence of $(T_n)$ on $D$. So $\mathcal{O}_g=0$ $\mu$-a.e.

Pick an $f\in X$. Now for any $\eta>0$, there is a $g\in D$ such that $\|f-g\|_X<\eta$ and \begin{eqnarray*}\mu(\{\omega: \mathcal{O}_f(\omega)>\delta\})&\leq &\mu(\{\omega: \mathcal{O}_{f-g}(\omega)>\delta\}).\end{eqnarray*} Furthermore, by the definition of the oscillation, $\mathcal{O}_f(\omega)\leq 2T(f)(\omega)$ a.e. Similarly for $\mathcal{O}_{f-g}$.
Hence \begin{eqnarray*}\mu(\{\omega: \mathcal{O}_f(\omega)>\delta\})&\leq &\mu(\{\omega: 2T(f-g)(\omega)>\delta\})\\&=&m(2T(f-g),\delta).\end{eqnarray*} As $T$ is of weak-type $(X,Y)$, $\|2T(f-g)\|_{M^*(\varphi_Y)}\leq 2\beta\|f-g\|_X<2\beta\eta$ where $\beta$ depends only on $T$. Rewriting this using Lemma \ref{lem2p22}, \begin{eqnarray*}\sup_{s>0}s\varphi_Y(m(2T(f-g),s))\leq 2\beta\eta.\end{eqnarray*} In particular, $\delta\varphi_Y(m(2T(f-g),\delta))\leq 2\beta\eta.$ Therefore $$\varphi_Y(\mu(\{\omega: \mathcal{O}_f(\omega)>\delta\}))\leq\frac{2\beta\eta}{\delta}.$$ As $\eta$ is arbitrary, $\displaystyle\varphi_Y(\mu(\{\omega: \mathcal{O}_f(\omega)>\delta\}))=0.$ Because a fundamental function is $0$ only at the origin, $\mu(\{\omega: \mathcal{O}_f(\omega)>\delta\})=0.$ Because $\delta$ is arbitrary, $\mathcal{O}_f=0$ $\mu$-a.e.\@ which implies that $(T_nf)$ does indeed converge $\mu$-a.e.\end{proof} 

The fundamental result towards completing step (2) of the three-step programme is given in Proposition \ref{theo2}. From this, many interesting pointwise ergodic theorems can be deduced, given further information on the nature of $X$. We start by constructing subsets $D$ of $X$ for whose elements a.e.\@ convergence is easy to check. To this end, for any $f\in L^1(G)$ and $x\in X$, we define \begin{equation}\label{E:actiononx}\alpha_f(x)=\int_G\alpha_t(x) f(t)~dt,\end{equation} where the integral is a Bochner integral. Because the action of $G$ on $(\Omega,\mu)$ is measure-preserving, on any r.\@i.\@BFS the automorphism $\alpha_t$ is an isometry and so $\alpha_f(x)\in X$ too. 

Note that the above equation actually gives a bounded bilinear mapping from $L^1(G)\times X$ into $X$, given by $(f,x)\mapsto\alpha_f(x)$.

\begin{definition}\label{D:D_X}Let $Y$ be a set of measurable functions on $(\Omega,\mu)$ and $\mathcal{L}\subseteq L^1(G)$. Define $$D_X(Y,\mathcal{L})=\{\alpha_f(x): f\in\mathcal{L},~x\in X\cap Y\},$$ which is a subset of $X$. In particular, if $\mathcal{F}_0$ consists of those integrable functions on $G$ with support of finite measure, we shall simply write $D_X$ for $D_X(L^{\infty}(\Omega),\mathcal{F}_0)$.\end{definition}

\begin{proposition}\label{theo2}  Let $(\Omega,\mu,G,\alpha)$ be a dynamical system and $(T_n)$ a sequence of convolution operators given by (\ref{E:p38}).  Suppose that the sequence $\big(\int_G k_n(t)~dt\big)$ converges and that $(k_n*\phi)$ converges weakly in $L^1$ for all $\phi\in L^1(G)$ with support of finite measure.

Then given a r.\@i.\@BFS $X$, the sequence $(T^{\#}_nf)$ converges a.e.\@ for every $f\in D_X$.\end{proposition}

\begin{proof} We begin by describing $T^{\#}_n$ explicitly, using the construction of Proposition \ref{P:concreteRealisation}. Let $f\in X$. For almost every $\omega\in\Omega$, the function $t\mapsto f(\alpha_t\omega)$ is locally integrable by Lemma \ref{L:F} and so by definition of $T_n$ we have $$T_nf(\alpha_t\omega)=\int_Gk_n(s)f(\alpha_{ts^{-1}}\omega)~ds.$$ Because we have assumed that $k_n$ is bounded and has bounded support, the integral converges for any locally integrable $f$, and in particular for any $f\in X$. Setting $t=1$, we obtain \begin{equation}\label{eHashdef}T^{\#}_nf(\omega)=\int_Gk_n(s)f(\alpha_{s^{-1}}\omega)~ds.\end{equation}

We prove that for any $f\in D_X$, the sequence $(T^{\#}_n(f))$ converges a.e. By definition of $D_X$, there exists a $g\in L^{\infty}(\Omega)\cap X$ and $\psi\in L^1(G)$ with support of finite measure, such that $f=\alpha_{\psi}(g)$. We compute: \begin{eqnarray*}T^{\#}_nf(\omega)&=&\int_Gk_n(t)f(\alpha_{t^{-1}}\omega)~dt\\&=&\int_G k_n(t)\int_G g(\alpha_{st^{-1}}\omega)\psi(s)~ds~dt\\&=&\int_Gg(\alpha_u\omega)\int_Gk_n(t)\psi(ut^{-1})~dt~du.\end{eqnarray*}

As the inner integrals converge weakly in $L^1(G)$, and bearing in mind that $g\in L^{\infty}(\Omega)$, we have proved that $(T^{\#}_n(f))$ converges a.e.\end{proof}

In the light of Propositions \ref{P:osc} and \ref{theo2}, to complete the three-step programme and hence prove pointwise ergodic theorems, we indicate situations where we can use the space $D_X$ to construct dense subsets of $X$.

On specialising to the case where the induced action $t\to \alpha_t$ of $G$ on $L^\infty(\Omega)$ is point-weak* continuous, we are able to show the existence of such a set for a very general class of Orlicz spaces. To see that this is indeed a mild restriction, notice that if $(\Omega,\mu)$ is a $\sigma$-finite standard Borel space and that $\alpha_t(\omega)=\omega$ for almost every $\omega\in \Omega$ only when $t=e$, then the measurability criterion in Definition \ref{D:dynsys} automatically implies point-weak* continuity (see \cite[Proposition XIII.1.1]{ta3}).  

\begin{theorem}\label{theoBirk} In the setup of Proposition \ref{theo2}, suppose that $X$ is an Orlicz space for which the Young's function $\Phi_X$ satisfies $\Delta_2$ globally and that $T^\#$ has weak type $(X,Y)$. Then $(T^{\#}_nf)$ converges a.e.\@ for every $f\in X$.\end{theorem}

\begin{proof} By hypothesis, $T^\#$ is of weak type $(X,Y)$. To prove the theorem we merely need to show that under the given hypothesis the set $D_X$ is dense in $X$. Once this is done we may execute steps (2) and (3) of our three step programme by applying Proposition \ref{P:osc} to the conclusion of  Proposition \ref{theo2}.

It therefore remains to show that $D_X$ is dense in $X$ when $\Phi_X\in \Delta_2$. We will prove this in two stages. We first show that all the simple functions with support of finite measure belong to the norm closure of $D_X$, and then in stage 2 show that these simple functions are dense in $X$. 

To complete the first stage of the proof, it is enough to show that any characteristic function of a measurable subset $E$ of $\Omega$ with finite measure, is in the closure of $D_X$. Let $E$ be such a set. By \cite[Proposition 1.6, Remark 1.7]{ar}(cf. \cite[Theorem III.3.2.2]{bla}), the continuity assumption on the action of $G$ ensures that the action $t\to\alpha_t(f)$ is point-norm continuous for all $f\in L^1(\Omega)$. Hence for $f=\chi_E$, we will have that $\alpha_t(\chi_E)\to \chi_E$ in $L^1$-norm as $t\to e$. Next pick a decreasing sequence of compact neighbourhoods $V_m$ in $G$ with intersection $e$, with $\sup_{t\in V_m}\|\alpha_t(\chi_E)-\chi_E\|_{L^1}\leq\frac{1}{m}$. Now let $\psi_m(t)= \frac{1}{|V_m|}\chi_{V_m}$ ($m\in \mathbb{N}$). By definition we will then have that $(\alpha_{\psi_m}(\chi_E))\subset D_X$. If we can show that this sequence convergences to $\chi_E$ in the $X$-norm, we will have that $\chi_E$ belongs to the closure of $D_X$ in $X$, as required. Since $\Phi_X$ is convex and $(V_m, \frac{1}{|V_m|}dh(t))$ a probability space, we may apply Jensen's inequality to see that 
\begin{eqnarray*}
\int_\Omega\Phi_X(|\alpha_{\psi_m}(\chi_E)-\chi_E|)\,d\mu&=&\int_\Omega\Phi_X\left(|\int_G\alpha_t(\chi_E)(\omega)\psi_m(t)\,dt - \chi_E(\omega)|\right)\,d\mu(\omega)\\
&\leq&\int_\Omega\Phi_X\left(\int_G|\alpha_t(\chi_E)(\omega) - \chi_E(\omega)|\psi_m(t)\,dt\right)\,d\mu(\omega)\\
&\leq&\int_\Omega\left(\int_G\Phi_X(|\alpha_t(\chi_E)(\omega) - \chi_E(\omega)|)\psi_m(t)\,dt\right)\,d\mu(\omega)\\
\end{eqnarray*}
It is now an exercise to see that for each fixed $\omega\in \Omega$ we have that $\int_G\Phi_X(|\alpha_t(\chi_E)(\omega) - \chi_E(\omega)|)\,d\mu=\Phi_X(1)\int_G|\alpha_t(\chi_E)(\omega) - \chi_E(\omega)|\,d\mu$ for each $t\in V_m$. This follows from the fact that the possible values of $|\alpha_t(\chi_E)(\omega) - \chi_E(\omega)|$ are 0 and 1, with $\Phi_X(t)=0$ if and only if $t=0$ since $\Phi_X\in \Delta_2$. Therefore \begin{eqnarray*}
\int_\Omega\Phi_X(|\alpha_{\psi_m}(\chi_E)-\chi_E|)\,d\mu &\leq& \Phi_X(1)\int_\Omega\left(\int_G|\alpha_t(\chi_E)(\omega) - \chi_E(\omega)|\psi_m(t)\,dt\right)\,d\mu(\omega)\\
&=&\Phi_X(1)\int_G \psi_m(t)\int_\Omega|\alpha_t(\chi_E)(\omega) - \chi_E(\omega)|\,d\mu(\omega)\,dt\\
&\leq& \Phi_X(1)\frac{1}{m}\int_G \psi_m(t)\,dt = \Phi_X(1)\frac{1}{m}.
\end{eqnarray*}
Thus $\int_\Omega\Phi_X(|\alpha_{\psi_m}(\chi_E)-\chi_E|)\,d\mu\to 0$ as $m\to \infty$. But since $\Phi_X\in \Delta_2$, it follows from \cite[Theorem 3.4.12]{raore1} that this convergence is equivalent to the norm convergence of $(\alpha_{\psi_m}(\chi_E))$ to $\chi_E$. This concludes the proof of stage 1.

It remains to show that the simple functions with support of finite measure are dense in $X$. This conclusion is however the content of \cite[Theorem 3.4.5]{raore1}.
\end{proof}

Finally, let us mention another way to obtain a dense subset of a r.\@i.\@BFS on which the pointwise convergence of the ergodic averages can readily be  checked.

\begin{corr}\label{corrInterpBirk}In the setup of Theorem \ref{theoBirk}, if $T^\#$ is also of weak type $(E,Z)$ for $E$ an Orlicz space with Young's function in the class $\Delta_2$ and $Z$ a r.i.\@ BFS, then whenever $X\cap E$ is dense in $X$, the sequence $(T^{\#}_nf)$ converges a.e.\@ for every $f\in X$.\end{corr}

\begin{proof} By Theorem \ref{theoBirk},  $(T^\#_nf)$ converges a.e. for all $f\in E$. Hence we have a dense subset of $X$, namely $X\cap E$, on which the ergodic averages converge pointwise. Applying Proposition \ref{P:osc} finishes the proof.\end{proof}

If for example $X=L^1(\Omega)$ and $E=Z=L^p(\Omega)$ for some $1<p<\infty$, this Corollary is applicable. A special case of this setup in given in \cite[Theorem 3]{ca}.

\section{Applications to the Orlicz spaces of Statistical Physics}\label{SP:appl}

In the recent paper by Labuschagne and Majewski \cite{lmsp} it was shown that the pair of Orlicz spaces $(L^{\cosh-1}, L\log(L+1))$ 
may well be better suited to the description of the statistics of large regular statistical systems (both classical and quantum) than the classical pairing of $(L^\infty, L^1)$. In support of this contention we recall that the space $L^{\cosh-1}$ provides a natural home for regular observables, whereas the states in $L\log(L+1)\cap L^1$ were all shown to have well-defined entropy. We close this paper by demonstrating the utility of our techniques for establishing the existence of pointwise convergence of ergodic averages on these spaces. To this end let $(\Omega,\mu,G,\alpha)$ be a dynamical system.

The spaces $L^{\cosh-1}(\Omega)$ and $L\log(L+1)(\Omega)$ are respectively determined by the Young's functions $\cosh(t)-1$ and 
$t\log(t+1)$. As was shown in the proof of \cite[Proposition 2.11]{lmsp}, the Young's functions $\cosh(t)-1$ and $t\log(t+1)$, are respectively equivalent to $(t+1)\log(t+1)-t$ and $e^t-t-1$. Hence the pair $(L^{\cosh-1}, L\log(L+1))$ are respectively isomorphic to the pair of Orlicz spaces generated by $\Phi(t)=(t+1)\log(t+1)-t$ and $\Psi(t)=e^t-t-1$. From the discussion on page 276 of \cite{raore1}, it is clear that $\Phi$ (and hence also $t\log(t+1)$) satisfies the $\Delta_2$ condition globally, and that these spaces are in general \emph{not} reflexive, with $L^\Psi(\Omega)$ equipped with the Orlicz norm appearing as the Banach dual of $L^\Phi(\Omega)$. Hence $L^{\cosh-1}(\Omega)$ is isomorphic to the Banach dual of $L\log(L+1)(\Omega)$. Having dealt with the necessary background we pass to the promised applications.

We first consider the space $L^{\cosh-1}(\Omega)$. Given any $2\leq p <\infty$, select $m\in \mathbb{N}$ so that $2m\leq p <2(m+1)$. Then of course $t^p\leq t^{2m}+t^{2(m+1)}$ for all $t\geq 0$. On considering the Maclaurin expansion of $\cosh(t)-1$, it now trivially follows that $\frac{1}{(2(m+1))!}t^p \leq\frac{1}{(2(m+1))!}(t^{2m}+t^{2(m+1)})\leq \cosh(t)-1$ for all $t\geq 0$. 

In the specific case where $G=\mathbb{R}$ and $k_n(t)=\frac{1}{n}\chi_{[-n,0]}$, the transfers of the operators $T_ng=k_n*g$ will in this case yield the ergodic 
averages $$T^{\#}_nf(\omega)=\int_{\mathbb{R}}k_n(s)f(\alpha_{-s}(\omega))\,ds=\frac{1}{n}\int_0^nf(\alpha_s(\omega))\,ds.$$But these are known to converge pointwise almost everywhere for any $f\in L^q(\Omega)$ where $1< q<\infty$. Since as we have just seen $L^{\cosh-1}(\Omega)\subset L^p(\Omega)$ for any $2\leq p<\infty$, we must trivially also have pointwise almost everywhere convergence of those same averages for $f\in L^{\cosh-1}(\Omega)$.

Next consider the space $L\log(L+1)(\Omega)$. On considering the limits of $t\log(t+1)/t^p$ as $t\searrow 0$ and $t\nearrow\infty$, it is an exercise to show that for any $1<p\leq 2$ there is some constant $K>0$ such that $t\log(t+1)\leq Kt^p$ for every $t\geq 0$. So for the Young's functions $\Phi_A(t)=\Phi_Y(t)=\Phi_W(t)=\sqrt{K}t^p$, and $\Phi_X(t)=\Phi_Z(t)=\Phi_{L\log(L+1)}(t)=t\log(t+1)$, we have that $$\Phi_Y(st)\geq \frac{1}{\sqrt{K}}\Phi_A(s)\Phi_X(t) \mbox{ and } \Phi_Z(st)\leq\Phi_A(s)\Phi_W(t)\mbox{ for all }t,s>0.$$With all the associated Orlicz spaces equipped with their Luxemburg norms, this can in turn be reformulated as the claim that
$$\varphi_A(t)\varphi_X(s)\leq\varphi_Y(\frac{1}{\sqrt{K}}st), \mbox{ and } \varphi_Z(st)\leq\varphi_A(s)\varphi_W(t)\mbox{ for all }t,s>0.$$So if additionally the group $G$ is amenable, then given $1<p\leq 2$, for any sequence $(T_n)$ of transferable operators satisfying the inequality
$$\|T_ng\|_{M^*(L^p)}\leq c_1\varphi_{L\log(L+1)}\left(c_0\int_G \Phi_{L\log(L+1)}(|g|)(t)\,dt\right)\qquad g\in L\log(L+1)(G)$$
for all $n\in \mathbb{N}$, it will follow from Theorem \ref{C:endcal2} that the operator $T^\#:=\sup_n |T^\#_n|$ will be of $L$-weak type $(L\log(L+1),L\log(L+1))$. Since $\Phi_X(t)=t\log(t+1)$ is known to satisfy $\Delta_2$ globally, obtaining almost everywhere convergence of the associated ergodic averages is a simple matter of applying Theorem \ref{theoBirk}.

\bibliographystyle{plain}
\bibliography{newmybib2}
\end{document}